\documentclass[final,a4paper]{article}
\usepackage[english]{babel}
\usepackage{showkeys}
\usepackage{graphicx}
\usepackage{amssymb,amsfonts,amsmath,amsthm,mathtools}
\usepackage{xcolor}
\usepackage{cite}
\usepackage[a4paper]{geometry}
\usepackage{filecontents}
\usepackage{algorithm}
\usepackage{algpseudocode}
\usepackage{tikz,pgfplots}

\newtheorem{theorem}{Theorem}
\newtheorem{lemma}[theorem]{Lemma}
\theoremstyle{remark}
\newtheorem{remark}[theorem]{Remark} 
\newtheorem{ex}[theorem]{Example}

\DeclarePairedDelimiter{\norm}{\lVert}{\rVert}
\DeclareMathOperator{\diag}{diag}

\makeatletter
\renewcommand*\env@matrix[1][c]{\hskip -\arraycolsep
  \let\@ifnextchar\new@ifnextchar
  \array{*\c@MaxMatrixCols #1}}
\makeatother

\begin{document}

\title{When is a matrix unitary or Hermitian plus low rank?\thanks{The research of the first three authors was partially supported by GNCS projects ``Analisi di matrici sparse e data-sparse: metodi numerici ed applicazioni''; the research of the first two was also supported by University of Pisa
		under the grant PRA-2017-05. The research of the fourth author was supported by the Research
		Council KU Leuven, project C14/16/056 (Inverse-free Rational Krylov Methods: Theory and Applications).}}

\author{Gianna M. Del Corso, Federico Poloni, Leonardo Robol, and\\ 
  Raf Vandebril\footnote{G.\ M.\ Del Corso and F.\ Poloni at University of Pisa, Dept. Computer Science, Italy. \texttt{gianna.delcorso@unipi.it},	\texttt{federico.poloni@unipi.it}. 
  L.\ Robol at University of Pisa, Dept. of Mathematics, and 
     Institute of Information Science and Technologies ``A. Faedo'', ISTI-CNR, Italy.
	 \texttt{leonardo.robol@unipi.it}. 
  R.\ Vandebril at University of Leuven (KU Leuven), Dept. Computer Science, Belgium.
	  \texttt{raf.vandebril@cs.kuleuven.be. }}}
\date{\today}

\maketitle

\begin{abstract}
  Hermitian and unitary matrices are two representatives of the class of normal matrices
  whose full eigenvalue decomposition can be stably computed in quadratic computing
  complexity once the matrix has been reduced, for instance,  to tridiagonal or Hessenberg form. Recently,  fast and reliable eigensolvers dealing with low rank
  perturbations of unitary and Hermitian matrices were proposed.  These structured eigenvalue problems
  appear naturally when computing roots, via confederate linearizations, of polynomials
  expressed in, e.g., the monomial or Chebyshev basis. Often, however, it is not known
  beforehand whether or not a matrix can be written as the sum of an Hermitian or unitary
  matrix plus a low rank perturbation.

  In this paper, we give necessary and sufficient conditions characterizing the class of Hermitian or
  unitary plus low rank matrices. The number of singular values deviating from
  $1$ determines the rank of a perturbation to bring a matrix to unitary form. A similar
  condition holds for Hermitian matrices; the eigenvalues of the skew-Hermitian part
  differing from $0$ dictate the rank of the perturbation. We prove that these relations
  are linked via the Cayley transform. 

 Then, based on these conditions, we identify the closest Hermitian or unitary plus rank $k$ matrix to a given matrix $A$, in Frobenius and spectral norm, and give a formula for their distance from $A$.
 Finally, we present a practical iteration to detect the low rank perturbation. Numerical tests prove that this straightforward algorithm is effective.

\end{abstract}


\section{Introduction}
Normal matrices  are computationally amongst the most pleasant matrices to
work with. The fact that their eigenvectors form a full orthogonal set is the basic
ingredient for developing many stable algorithms. Even though generic normal matrices
are less common in practice the unitary and (skew-)Hermitian matrices are prominent members.
Eigenvalue and system solvers for Hermitian \cite{b200,b135} and unitary matrices
\cite{AuMaVaWa14b,m281,p559} have been examined thoroughly and well-tuned
implementations are available in, e.g., \texttt{eiscor}\footnote{
	EISCOR: eigenvalue solvers based on unitary core transformations. \texttt{https://github.com/eiscor/eiscor}.}
and LAPACK\footnote{
	LAPACK: linear algebra package.
	\texttt{https://www.netlib.org/lapack/}.}.
Some matrices are not exactly Hermitian or unitary, yet a low rank perturbation of
those. These matrices are the subject of our study: our aim is to provide a theoretical
characterization of these matrices in terms of their singular- or eigenvalues.

It has been noted that several linearizations of (matrix) polynomials are 
low rank perturbations of unitary or Hermitian matrices (see, e.g., 
\cite{bini2016class,de2013condition,delcorso2018factoring,robol2015exploiting,robol2017framework}
and the references therein). Computing
the roots of these polynomials hence coincides with computing the eigenvalues of the
associated matrices. 
For instance, when computing roots of
polynomials expressed in bases that admit a three-terms recurrence such as
the Chebyshev one, one ends up with the Comrade matrix, which is
symmetric plus rank $1$. Algorithms to solve this problem were developed by
Chandrasekaran and Gu
\cite{n842}; Delvaux and Van Barel \cite{d047}; Del Corso and Vandebril \cite{VaDe09b};
and, one the fastest and most reliable ones is due to Eidelman, Gemignani, and Gohberg
\cite{q763}.
When studying orthogonal
polynomials on the unit circle \cite{b164,b156} we end up with
unitary-plus-low-rank matrices. The companion matrix, whose eigenvalues coincide with the roots of a polynomial in
the monomial basis, is the most popular case. Various algorithms differing with respect to
storage-scheme, compression, explicit or implicit QR algorithms were proposed. We cite few
and refer to the references therein for a full overview: Bini,
Eidelman, Gemignani, and Gohberg \cite{p980}; Van Barel, Vandebril, Van Dooren, and
Frederix \cite{VBVaVDFr08}; Chandrasekaran and Gu \cite{q232}; Boito, Eidelman, and
Gemignani \cite{BoEiGe13} Bevilacqua, Del Corso, Gemignani~\cite{BDG15}; and, a fast and provably stable version, is presented in
the book of Aurentz, Mach, Robol, Vandebril, and Watkins \cite{AuMaRoVaWa18}. 
Extensions for efficiently handling corrections with larger rank, 
necessary to deal with block companion matrices, 
have been presented by Aurentz et al.\
\cite{aurentz2016fast}; Bini and Robol \cite{bini2015quasiseparable}; Gemignani and Robol \cite{gemignani2017fast}; Bevilacqua, Del
Corso, and Gemignani \cite{BDG18}; and, Delvaux \cite{d047}.

In the context of Krylov methods the structure of unitary and Hermitian plus low rank
matrices can be exploited as well. Not only it provides a fast matrix-vector
multiplication, but also the structure of the Galerkin projection profits from to design
faster algorithms.  Huckle \cite{r047} and Huhtanen (see \cite{r058} and the references therein) analyzed the Arnoldi
and Lanczos method for normal matrices. Barth and Manteuffel \cite{BaMa00} examined
classes of matrices resulting in short recurrences and proved that matrices whose adjoint
is a low-rank perturbation of the original matrix allowed short CG recursions; a more
detailed analysis can also be found in the overview of Liesen and Strako\v{s} \cite{LiStra08}. It is
interesting to note that both unitary and Hermitian plus low rank matrices allow for such
recursions. 
Liesen examined in more detail the relation between a matrix and a
rational function of its adjoint \cite{q777}. An alternative manner to devise the short
recurrences was proposed by Beckermann, Mertens, and Vandebril \cite{BemeVa18}. An
algorithm to exploit the short recurrences to develop efficient solvers for Hermitian plus
low rank case is the progressiver GMRES method, proposed by Beckermann and Reichel
\cite{q882}, this method was tuned later on and stabilized by Embree et al. \cite{Em12}.

In this article we characterize unitary and Hermitian plus low rank matrices by
examining their singular- and eigenvalues. We prove that a
matrix having at most $k$
singular values less than $1$ and at most $k$ greater than $1$ is
unitary plus rank $k$. Similarly, by examining the eigenvalues of the skew-Hermitian 
part of a matrix we show that if at most $k$ of these eigenvalues are greater than $0$ and at most $k$ are
smaller than zero, the matrix is Hermitian plus rank $k$.
These characterizations enable us to determine the closest unitary or Hermitian plus rank $k$
matrices in the spectral and Frobenius norms
by setting some well-chosen singular- or eigenvalues  to
$1$ or $0$. We also show that the Cayley transform bridges between the Hermitian and
unitary case. As a proof of concept, we designed and tested a straightforward Lanczos based algorithm to
detect the low-rank part in some test cases.

The article is organized as follows. In Section~\ref{sec:prel} we revisit some
preliminary results. Section~\ref{sec:uni} discusses
necessary and sufficient conditions for a matrix to be unitary plus low rank. 
Constructive proofs furnish the closest unitary plus low rank matrix in spectral and
Frobenius norm. Sections~\ref{sec:herm} and~\ref{sec:dherm} discuss 
the analogue of these results for the Hermitian plus low rank structure.
The Cayley transform, Section~\ref{sec:cayley},
allows us to transform the unitary into the Hermitian problem. 
Algorithms to extract the low rank part from a Hermitian (unitary) plus low rank matrix and some experiments are proposed in Sections~\ref{sec:lanczos} and~\ref{sec:experiments}. We conclude in Section~\ref{sec:conc}.

\section{Preliminaries}
\label{sec:prel}

In this text we make use of the following conventions. The symbols $I$ and $0$ denote the
identity and zero matrix, and may have subscripts denoting their size whenever that is not
clear from the context.  We use $\sigma_1(M) \geq \sigma_2(M) \geq \dots \geq
\sigma_{n}(M)$ to denote the singular values of a matrix $M\in\mathbb{C}^{n\times n}$, and
$\lambda_1(H) \geq \lambda_2(H) \geq \dots \geq \lambda_{n}(H)$ stand for the eigenvalues
of a Hermitian matrix $H\in\mathbb{C}^{n\times n}$. We use the $\diag$ operator which
stacks its arguments, which could be scalars, matrices, or vectors, in a block diagonal matrix (possibly with non-square blocks on its diagonal).

The following results are classical. We rely on them in the forthcoming proofs and for
completeness we have included them.

\begin{theorem}[\protect{Weyl's inequalities, \cite[Theorem~4.3.16 and~Exercise 16, page
423]{HornJohnson}}] \label{thm:hj1} For every pair of matrices $M,N\in\mathbb{C}^{n\times
n}$ and for every $i,j$ such that $i+j\leq n+1$,
\[ \sigma_{i+j-1}(M\pm N) \leq \sigma_i(M) + \sigma_j(N).
\] If $M$, $N$ are Hermitian, then the same inequality holds for their eigenvalues.
\[ \lambda_{i+j-1}(M\pm N) \leq \lambda_i(M) + \lambda_j(N).
\]
\end{theorem}

\begin{theorem}[\protect{Interlacing inequalities, \cite[Theorems~4.3.4]{HornJohnson}} and \cite{Thompson}] 
\label{thm:hj2}
Let $M\in\mathbb{C}^{n\times n}$ and $N\in\mathbb{C}^{n\times (n-k)}$ be
a submatrix of $M$ obtained by removing $k$ columns from it. Then,
\[ \sigma_{i+k}(M) \leq \sigma_i(N) \leq \sigma_i(M).
\]
In the Hermitian case we get similar inequalities.
 Let $M \in \mathbb{C}^{n\times n}$ be Hermitian and
$N\in\mathbb{C}^{(n-k)\times(n-k)}$ be a (Hermitian) principal submatrix of $M$. Then,
\[ \lambda_{i+k}(M) \leq \lambda_i(N) \leq \lambda_i(M).
\]
\end{theorem}

Moreover, recall that $M\in\mathbb{C}^{n\times n}$ is unitary if and only if
$\sigma_i(M)=1$, $\forall i=1,\ldots, n$.

\section{Detecting  unitary-plus-rank-$k$ matrices}
\label{sec:uni}

We call $\mathcal{U}_k$ the set of unitary-plus-rank-$k$ matrices, i.e., $A\in
\mathcal{U}_k$ if and only if there exists a unitary matrix $Q$ and two 
 skinny matrices $G,B\in\mathbb{C}^{n\times k}$ such that $A = Q + GB^*$. 
This implies that $\mathcal{U}_k \subseteq \mathcal{U}_{k+1}$ for any $k$. 

\begin{theorem} \label{thm:mainortho}
Let $A \in \mathbb{C}^{n\times n}$ and $0 \leq k \leq n$.
Then, $A \in \mathcal{U}_k$ if and only if $A$ has at most $k$ singular values strictly greater than $1$ and at most $k$ singular values strictly smaller than $1$.
\end{theorem}

Before proving this result, we point out that looking at the singular values is a good
guess, since 
being unitary-plus-rank-$k$
is invariant under unitary equivalence transformations.

\begin{remark} \label{rem:conjugation}
$A= Q + GB^* \in \mathcal{U}_k$ if and only if $U^* A V \in \mathcal{U}_k$ for any unitary matrices $U,V$. Indeed,
we have that
\[U^* A V = \underbrace{U^* Q V}_{\hat{Q}} + \underbrace{U^* G}_{\hat{G}} \underbrace{B^*
  V}_{\hat{B}^*}\in \mathcal{U}_k.
\]
\end{remark}

We start proving a simple case ($n=2,k=1$) of Theorem~\ref{thm:mainortho}, which will
act as a building block for the general proof.

\begin{lemma} \label{lem:2x2}
For every pair of real numbers $\sigma_1$ and $\sigma_2$ such that $\sigma_1 \geq 1 \geq \sigma_2
\geq 0$, we have
\[
\begin{bmatrix}
    \sigma_1 & 0\\
    0 & \sigma_2
\end{bmatrix} \in \mathcal{U}_1.
\]
\end{lemma}
\begin{proof}
We prove that the diagonal $2 \times 2$ matrix can be decomposed as 
a plane rotation plus a rank $1$ correction. In particular, 
we look for $c,s,a,b \geq 0$ such that
\[
	\begin{bmatrix}
	    \sigma_1 & 0\\ 0 & \sigma_2
	\end{bmatrix} =
	\begin{bmatrix}[r]
	    c & s\\
	    -s & c
	\end{bmatrix}
	+
	\begin{bmatrix}
	    a & -s\\
	    s & -b
	\end{bmatrix},
\]
i.e., $\sigma_1 = c+a$ and $\sigma_2 = c-b$. In addition, we impose that
the first summand is unitary (i.e., $c^2+s^2=1$), and that the second has rank 1 (i.e., $s^2=ab$). A simple computation shows that
\[
c = \frac{\sigma_1 \sigma_2+1}{\sigma_1 + \sigma_2}, \quad a =
\frac{\sigma_1^2-1}{\sigma_1+\sigma_2}, \quad b = \frac{1-\sigma_2^2}{\sigma_1 +
  \sigma_2},\;\; \mbox{and}\;\;  s = \sqrt{ab}
\]
satisfy these conditions.
\end{proof}

We are now ready to prove Theorem~\ref{thm:mainortho}.

\begin{proof}[Proof  of Theorem~\ref{thm:mainortho}]

First note that the case $k=n$ is trivial
$\mathcal{U}_{n}=\mathbb{C}^{n\times n}$.
So we assume $k<n$. The conditions on the singular values can be written as  two inequalities 
\begin{equation} \label{inequalities}
1 \geq \sigma_{k+1}(A) \quad \text{and} \quad \sigma_{n-k}(A) \geq 1.
\end{equation}
\begin{itemize}
\item 
  We first show that if $A\in\mathcal{U}_k$ then the inequalities~\eqref{inequalities}
  hold. Suppose that $A=Q + GB^*$. Then, by Theorem~\ref{thm:hj1},
  \[
  \sigma_{k+1}(A) \leq \sigma_1(Q) + \sigma_{k+1}(GB^*) = 1 + 0 = 1.
  \]
  And, again following from Theorem~\ref{thm:hj1},
  \[
  1 = \sigma_n(Q) \leq \sigma_{n-k}(A) + \sigma_{k+1}(GB^*) = \sigma_{n-k}(A).
  \]

\item We now prove the reverse implication, i.e., if the two
  inequalities~\eqref{inequalities} hold then $A\in\mathcal{U}_k$.  To simplify things, we
  introduce $k_{-}$ denoting the number of singular values smaller than $1$ and $k_{+}$
  standing for the number of singular values larger than $1$.  The conditions state that
  $\ell=\max\{k_{-},k_{+}\}\leq k$.  We will prove that $A\in\mathcal{U}_{\ell}$. Note
  that $\mathcal{U}_\ell\subseteq\mathcal{U}_k$.
  Let $h=\min\{k_{-},k_{+}\}$.

We reorder the diagonal elements of $\Sigma$ to group the singular values into separate
diagonal blocks of  three types:
  \begin{itemize}
  \item Diagonal blocks $\Sigma_1,\Sigma_2,\dots,\Sigma_{h}$ of size $2\times 2$,
    containing each one singular value larger than $1$ and one smaller than $1$. Since
    $h=\min\{k_{-},k_{+}\}$ either all singular values smaller than $1$ or all singular
    values larger than $1$ are incorporated in these blocks.
  \item Diagonal blocks $\Sigma_{h+1},\dots,\Sigma_{\ell}$ of size $1\times 1$ containing
    the remaining singular values different from $1$. Note that all these blocks will contain
    either singular values that are larger than $1$ or smaller than $1$, depending on
    whether $h=k_{-}$ or $h=k_{+}$. In case
    $k_{-}=k_{+}=h=\ell$ there will be no blocks of this type.
  \item One final block equal to the identity matrix of size $m = n-h-\ell=n-k_{-}-k_{+}$, containing
    all the singular values equal to $1$.
  \end{itemize}
  For example,for $A$ having singular values $(3,2,2,1.3,1.2,1,1,1,0.6,0.2)$,  we
  can take $\Sigma_1=\diag(3,0.2)$, $\Sigma_2=\diag(2,0.6)$, $\Sigma_3=2$, $\Sigma_4=1.3$,
  $\Sigma_5=1.2$, and $\Sigma_6=I_3$.

  Clearly, this decomposition always exists, and although it is not unique, the number and
  types of blocks are. 

  For each $i=1,2,\dots,\ell$, the matrix $\Sigma_i$ is unitary plus rank $1$. This follows
  from  Lemma~\ref{lem:2x2} for $2\times 2$ blocks, and is trivial for the $1\times 1$ blocks.
Hence for each $i$ we can write $\Sigma_i = Q_i + g_i b_i^*$,
  where $Q_i$ is unitary and $g_i,b_i$ are vectors. So we have
  \[
  \diag(\Sigma_1,\Sigma_2,\dots,\Sigma_\ell,I_m) = \diag(Q_1,Q_2,\dots,Q_\ell,I_m) + GB^*,
  \]
  with
  \[
  G = \begin{bmatrix}
    \diag(g_1,g_2,\dots,g_\ell)\\
    0_{m\times \ell}
  \end{bmatrix}, \quad\mbox{ and }\quad B = \begin{bmatrix}
    \diag(b_1,b_2,\dots,b_\ell)\\
    0_{m\times \ell}
  \end{bmatrix}.
  \]
  Therefore $\diag(\Sigma_1,\Sigma_2,\dots,\Sigma_\ell,I) \in \mathcal{U}_{\ell}$. 
Since this
  matrix can be obtained from a unitary equivalence on
$A$ (singular value decomposition and a permutation) we have 
  $A \in \mathcal{U}_\ell$. 
\end{itemize}\end{proof}

\begin{ex}
The matrix $U\operatorname{diag}(3,2,1,1,1,0.5)V^*$ belongs to $\mathcal{U}_2$ (but not to $\mathcal{U}_1$). The matrix $U\operatorname{diag}(5,0.4,0.3,0.2)V^*$ belongs to $\mathcal{U}_3$ (but not to $\mathcal{U}_2$). The matrix $5I_4$ belongs to $\mathcal{U}_4$ (but not to $\mathcal{U}_3$).
\end{ex}

\section{Distance from unitary plus rank $k$}
\label{sec:duni}

Theorem~\ref{thm:mainortho} provides an effective criterion to characterize matrices in
$\mathcal U_k$ based on the singular values. One of the main features of the singular
value decomposition is that it automatically provides the optimal rank $k$ approximation
of any matrix, in the sense of the $2$-and the Frobenius norm. In this section, we show
that the criterion of Theorem~\ref{thm:mainortho} can be used to compute the best
unitary-plus-rank-$k$ approximant for any value of $k$.

The problem is thus to find a matrix in $\mathcal{U}_k$ that minimizes the distance to a
given matrix $A$.
This can be achieved by setting the supernumerary singular values preventing the
inequalities~\eqref{inequalities} from being satisfied to $1$.

\begin{theorem} \label{thm:distanceUk}
Let the matrix $A\in\mathbb{C}^{n\times n}$ have singular value decomposition $U\Sigma V^*$, and 
denote by $\hat A= U \widehat{\Sigma} V^*$,
where $\widehat{\Sigma}$ is the diagonal matrix with elements
\[
\hat{\sigma}_i = \begin{cases}
1 & \mbox{ if } k< i \leq k_+, \mbox{ or }  n-k_- < i \leq n-k,\\
\sigma_i & \text{otherwise}
\end{cases}
\]
with $k_+$ (resp. $k_-$)  the number of singular values of $A$ strictly greater
(resp. smaller) than $1$, we have that
\[
\min_{X \in \mathcal{U}_k} \norm{A - X}_2 = \norm{A- \hat A}_2= \max\{0, \sigma_{k+1} - 1, 1-\sigma_{n-k}\}.
\]

\end{theorem}

\begin{proof}
The matrix $\hat A = U \widehat{\Sigma} V^*$ clearly belongs to $\mathcal{U}_k$ by
Theorem~\ref{thm:mainortho}, hence $\norm{A-U\widehat{\Sigma}V^*}_2 = \norm{\Sigma - \widehat{\Sigma}}_2 = \max\{\sigma_{k+1} - 1, 1-\sigma_{n-k}, 0\}$.
It remains thus to
prove that for every 
$X\in\mathcal{U}_k$ one has $\norm{A-X}_2 \geq \norm{A - \hat A}_2$. 

By Weyl's inequality (Theorem~\ref{thm:hj1}), we have that
$\sigma_{k+1}(A) \leq \sigma_1(A - X) + \sigma_{k+1}(X)$, 
and therefore 
\[
  \norm{A - X}_2 \geq  \sigma_{k+1}(A) - \sigma_{k+1}(X)\ge \sigma_{k+1}(A)-1. 
\]
Similarly from $\sigma_{n-k}(X) \leq \sigma_1(A - X) + \sigma_{n-k}(A)$, we have
\[
\norm{A - X}_2 \geq  \sigma_{n-k}(X) - \sigma_{n-k}(A)\ge 1-\sigma_{n-k}(A). 
\]
We obtain 
\[
  \norm{A - X}_2 \geq \max \{0, \sigma_{k+1} - 1, 1 - \sigma_{n-k} \} = \norm{A - \hat A}_2.
\]
\end{proof}

Theorem~\ref{thm:distanceUk} provides a deterministic construction 
for a minimizer of $\norm{A - X}_2$, but this minimizer is not unique. 
This is demonstrated in the following example. 

\begin{ex}
Let us consider, for arbitrary unitary matrices $U, V$, the matrix $A$ defined as
\[
  A = U \begin{bmatrix}
  2 \\
  & 1.5 \\
  && 1 \\
  &&& 1 \\
  &&&& .5 \\
  \end{bmatrix}V^*. 
\]
We know, from  Theorem~\ref{thm:mainortho} that $A \in \mathcal U_2$. 
We want to determine the distance 
from $A$ to $\mathcal U_1$. Theorem~\ref{thm:distanceUk} yields the
approximant $\hat A$ defined as follows: 
\[
  \hat A := U \begin{bmatrix}
    2 \\
  & 1 \\
  && 1 \\
  &&& 1 \\
  &&&& .5 \\
  \end{bmatrix} V^*, \qquad 
  \norm{A - \hat A}_2 = 0.5. 
\]
However, the solution is not unique. For instance, the family of matrices determined as 
\[
  \hat A(t, s) := U \begin{bmatrix}
    2 + t \\
    & 1 \\
    && 1 \\
    &&& 1 \\
    &&&& .5 + s
  \end{bmatrix} V^*
\]
still lead to $\norm{A - \hat A(t,s)}_2 = 0.5$ for $t, s \in [ -0.5, 0.5 ]$. 
\end{ex}

The same matrix is a minimizer also in the Frobenius norm.
\begin{theorem}
	Let the matrix $A\in\mathbb{C}^{n\times n}$ have singular value decomposition $U\Sigma V^*$ and denote by $\hat{A}=U \widehat{\Sigma} V^*$
	where $\widehat{\Sigma}$ is the diagonal matrix with  elements
	\[
	\hat{\sigma}_i = \begin{cases}
	1 & \mbox{ if } k< i \leq k_+, \mbox{ or }  n-k_- < i \leq n-k,\\
	\sigma_i & \text{otherwise}
	\end{cases}
	\]
	where $k_+$ (resp. $k_-$) is the number of singular values of $A$ which are strictly
	greater (resp. smaller) than $1$. Then we have 
		\[
\min_{X \in \mathcal{U}_k} \norm{A - X}_F = \norm{A-\hat A}_F,
	\]
	and moreover
	\begin{equation}
	\label{eq:norm:frob}
	\|A-\hat{A}\|_F^2 = \sum_{i=k+1}^{k_+} (\sigma_i-1)^2 + \sum_{i=n-k_-+1}^{n-k} (\sigma_i-1)^2.
	\end{equation}

\end{theorem}

\begin{proof} Since the Frobenius norm is unitarily invariant, we immediately have \eqref{eq:norm:frob}.
To complete the proof, we show that for an arbitrary $X\in\mathcal{U}_k$ we have $\|A-X\|_F\geq
\|A-\hat{A}\|_F$.
Let $\Delta$ be the matrix such that $X=A + \Delta \in \mathcal{U}_k$. We
will prove that $\|\Delta\|_F^2 \geq \|A-\hat{A}\|_F^2$.
Consider $\tilde{\Delta} = U^*\Delta V$ and partition
\[
U^* X V = U^*(A+\Delta)V = \begin{bmatrix}
    \Sigma_1 + \tilde{\Delta}_1 & \Sigma_2 + \tilde{\Delta}_2 & \Sigma_3 + \tilde{\Delta}_3
\end{bmatrix},
\]
where $\Sigma_1 \in \mathbb{C}^{n\times k_+}$ contains the singular values of $A$ which
are larger than $1$, $\Sigma_2$ the singular values equal to $1$, and $\Sigma_3 \in \mathbb{C}^{n\times k_-}$ the singular values smaller than $1$. 

Since $U^*(A+\Delta )V \in \mathcal{U}_k$, we have
\[
\sigma_{k+1}(U^*(A+\Delta)V) \leq 1, \quad \sigma_{n-k}(U^*(A+\Delta)V) \geq 1.
\]
\begin{itemize}\item 
  Assume first that $k_+ > k$. By the interlacing inequalities (Theorem~\ref{thm:hj2}),
  \[
  \sigma_{k+1}(\Sigma_1 + \tilde{\Delta}_1) \leq \sigma_{k+1}(U^*(A+\Delta)V) \leq 1,
  \]
  and then by Weyl's inequalities (Theorem~\ref{thm:hj1}) for each $k < i \leq k_+$
  \[
  \sigma_i = \sigma_i(\Sigma_1) \leq \sigma_{k+1}(\Sigma_1 + \tilde{\Delta}_1) +
  \sigma_{i-k}(\tilde{\Delta}_1) \leq 1 + \sigma_{i-k}(\tilde{\Delta}_1),
  \]
  from which we obtain $\sigma_{i-k}(\tilde{\Delta}_1) \geq \sigma_i-1$, and hence
  \begin{equation} \label{ineq11} \norm{\tilde{\Delta}_1}_F^2 \geq \sum_{i=1}^{k_+-k}
    \sigma_{i}(\tilde{\Delta}_1)^2 \geq \sum_{i=k+1}^{k_+} (\sigma_i-1)^2.
  \end{equation}
  Note that~\eqref{ineq11} holds trivially also when $k_+ \leq k$.

  \item Similarly, assume that $k_- > k$ (the case $k_- \leq k$ is trivial), and we use
    interlacing inequalities (Theorem~\ref{thm:hj2}) 
  to get
  \[
  \sigma_{k_- - k}(\Sigma_3 + \tilde{\Delta}_3) \geq \sigma_{n-k}(U^*(A+\Delta)V) \geq 1,
  \]
  and Weyl's inequalities (Theorem~\ref{thm:hj1}) for each $k<i \leq k_-$ to show
  \[
  1 \leq \sigma_{k_- - k}(\Sigma_3 + \tilde{\Delta}_3) \leq \sigma_{k_- +1-i}(\Sigma_3) +
  \sigma_{i-k}(\tilde{\Delta}_3) = \sigma_{n+1-i} + \sigma_{i-k}(\tilde{\Delta}_3),
  \]
  from which we obtain $\sigma_{i-k}(\tilde{\Delta}_3) \geq 1 - \sigma_{n+1-i}$. Hence
  \begin{equation} \label{ineq12} \norm{\tilde{\Delta}_3}_F^2 \geq \sum_{j=n+1-k_-}^{n-k}
    (1-\sigma_j)^2.
  \end{equation}
\end{itemize}
Putting together~\eqref{ineq11} and~\eqref{ineq12}, we have
\[
\norm{\Delta}_F^2 \geq \norm{\tilde{\Delta}_1}_F^2 + \norm{\tilde{\Delta}_3}_F^2 \geq \sum_{i=k+1}^{k_+} (\sigma_i-1)^2 + \sum_{j=n+1-k_-}^{n-k} (1-\sigma_j)^2 = \|A-\hat{A}\|_F^2,
\]
which is precisely what we wanted to prove.
\end{proof}

In this case, unlike in the $2$-norm setting, this minimizer is unique when all
 singular values are different.\footnote{In fact the constraint of all singular values differing can be relaxed: one can construct examples in which there are several identical singular values and all (or none) of them should be changed to 1.}

\section{Detecting Hermitian-plus-rank-$k$ matrices}
\label{sec:herm}

We call $\mathcal{H}_k$ the set of Hermitian-plus-rank-$k$ matrices, i.e., $A\in\mathcal{H}_k$ if and only
if there exist a Hermitian matrix $H$ and two matrices $G,B\in\mathbb{C}^{n\times k}$ such that $A = H +GB^*$.

In this and the next section we will answer similar questions: how can we tell if $A\in \mathcal{H}_k$? How do we find the distance from a matrix $A$ to the closed set $\mathcal{H}_k$?
To answer these questions, we first need an Hermitian equivalent of Lemma~\ref{lem:2x2}. 

\begin{lemma} \label{lem2x2herm}
For any pair of real numbers $\sigma_1$ and $\sigma_2$,  such that $\sigma_1\geq 0, \sigma_2 \leq 0$,  there are two real vectors $c$ and $b$ such that
\begin{equation} \label{eq:hermU1}
\begin{bmatrix}
    \sigma_1\\
    & \sigma_2
\end{bmatrix} = bc^T+cb^T.
\end{equation}
\end{lemma}
\begin{proof}
	Define the vectors $c$ and $b$ as follows:
	\[
	  b := \frac{1}{2} \begin{bmatrix}
	    \sqrt{\sigma_1} \\ - \sqrt{-\sigma_2} \\
	  \end{bmatrix}, \qquad 
	  c := \begin{bmatrix}
	     \sqrt{\sigma_1} \\
	     \sqrt{-\sigma_2} \\
	  \end{bmatrix}. 
	\]
	The result follows by a direct computation. 
\end{proof}

The next Theorem is the analogue of Theorem~\ref{thm:mainortho}, 
where we look at eigenvalues of the skew-Hermitian part of a matrix instead of at 
the singular values. We rely on the following lemma.

\begin{lemma} \label{lem:eigssign} 
	Let $B, C$ be any $n \times k$ full rank matrices, and let 
	$S =BC^* +CB^*$. Then, $S$ has at most $k$ positive and at most $k$ negative
	eigenvalues. 
\end{lemma}

\begin{proof}
Up to a change of basis, we can assume $C = \begin{bmatrix}
    I_k\\0
\end{bmatrix}$. Then, $S$ has a trailing $(n-k) \times (n-k)$ zero submatrix $T$. 
 Let $\lambda_1(S)\ge  \cdots \ge  \lambda_n(S)$ be the eigenvalues of $S$.
By the interlacing inequalities, $\lambda_{k+1}(S) \leq \lambda_1(T) = 0$, and $\lambda_{n-k}(S) \geq \lambda_{n-k}(T) = 0$.
\end{proof}

\begin{theorem} \label{thm:mainherm}
Let $A\in\mathbb{C}^{n\times n}$, and $0\leq k \leq n$. Then, $A\in\mathcal{H}_k$ if and only if the Hermitian matrix $S(A):=\frac{1}{2i}(A-A^*)$ has at most $k$ positive eigenvalues  and at most $k$ negative eigenvalues.

\end{theorem}

\begin{proof}
  Let $\lambda_1, \dots, \lambda_n$ be the eigenvalues of $\frac{1}{2i}(A-A^*)$, sorted by
  decreasing order, i.e., $\lambda_1 \geq \lambda_2 \geq \dots \geq \lambda_n$. Then the
  stated condition is equivalent to demanding $\lambda_{k+1}\leq 0$, $\lambda_{n-k}\geq 0$.

We first show that if $A\in\mathcal{H}_k$ these inequalities hold. If $A\in\mathcal{H}_k$,
then there exists a Hermitian matrix $H$ and two matrices $G, B\in \mathbb{C}^{n\times k}$
such that $A=H+GB^*$. Then
\[  
\frac{1}{2i}(A-A^*)=\frac{1}{2i}(GB^*-BG^*)=CB^*+BC^*,
  \]
with $C=G/(2i)$. The result follows from Lemma~\ref{lem:eigssign}.

We now prove the converse, that is, each matrix  satisfying the inequalities belongs to
$\mathcal{H}_k$.
We first
prove that each Hermitian matrix $S$ that has $k_+$ strictly positive eigenvalues 
 and  $k_-$ strictly
negative eigenvalues, where $\ell=\max\{k_-,k_+\}\leq k$  can be written as $CB^* + BC^*$ with $B,C \in \mathbb C^{n \times \ell}$.

 Let us assume that the eigenvalues of $S$ are $\lambda_j$ with
\[
  \lambda_1, \ldots, \lambda_{k_+} > 0, \qquad 
  \lambda_{k_+ +1}, \ldots, \lambda_{k_+ + k_-} < 0, \qquad 
  \lambda_{k_+ + k_- + 1}, \ldots, \lambda_n = 0. 
\]

Since $S$ is 
normal, we can diagonalize it by an orthogonal transformation and obtain 
\[
  U^*SU = \mathrm{diag}(
    \Lambda_1,
    \ldots,
    \Lambda_{h},
   \Lambda_{h+1},
    \ldots,
    \Lambda_{\ell},
    0_m
  ),
\]
where we get, as in the proof of the unitary case, three types of blocks:
\begin{itemize}
\item Diagonal blocks $\Lambda_1, \Lambda_2, \ldots, \Lambda_{h}$ of size $2\times 2$
  containing one eigenvalue larger and one eigenvalue smaller than $0$.
\item $1\times 1$ matrices $\Lambda_{h+1}, \ldots, \Lambda_{\ell}$ containing the
  remaining eigenvalues differing from $0$. Since either all positive or negative
  eigenvalues are used already to form the blocks $\Lambda_1$ up to $\Lambda_h$, we end up with
  scalar blocks that all have the same sign.
\item A final zero matrix of size $m=n-k_{-}-k_+=n-2h-\ell$.
\end{itemize}

Lemma~\ref{lem2x2herm} tells us  that 
each of the $2 \times 2$ blocks $\Lambda_j$, for $j=1,\ldots,h$ can be written as $b_j c_j^* + c_jb_j^*$, for
appropriate choices of $b_j, c _j$,
 since the eigenvalues on the diagonal
are real and have opposite sign. Moreover, the remaining 
diagonal entries $\Lambda_j$ for $j=h+1,\ldots,\ell$ can be written choosing 
$b_j =  \lambda_j$ and $c_j = 1/2$. Therefore, we conclude that $Q^*SQ$ 
is of the form $\tilde C \tilde B^T + \tilde B \tilde C^T$ for 
some $\tilde C, \tilde B$ with $\ell$ columns. Setting $G := (-2i)Q \tilde C$ and 
$B := Q \tilde B$ proves our claim. 
It is immediate to verify that the matrix  $A - GB^*$ is 
Hermitian, since $(A -\, GB^*)^* - (A -\, GB^*) = 0$.
\end{proof}

Theorem~\ref{thm:mainherm} has alternative formulations as well. We can for instance look
at $A-A^*$ and count the number of eigenvalues with positive and negative
imaginary parts, since when $A$ is Hermitian, $iA$ will be skew-Hermitian.

We will not go into the details, but it is obvious that Theorem~\ref{thm:mainherm} 
admits an equivalent formulation to check whether a matrix is 
a rank $k$ perturbation of a skew-Hermitian matrix. To this end one 
 considers $A+A^*$ and counts the number of positive and negative eigenvalues.

\section{Distance from Hermitian plus rank $k$}

\label{sec:dherm}

In this section we will construct, given an arbitrary matrix $A$, the closest
Hermitian-plus-rank-$k$ matrix in both the $2$- and the Frobenius norm.
The following lemma comes in handy.

\begin{lemma}\label{lem:norm}
	Let $\norm{\cdot}$ be any unitarily invariant norm, and 
	$X\in \mathbb{C}^{n\times n}$, and $S(X)=\frac{X-X^*}{2i}$. Then, 
	$\|S(X)\|\le \|X\|$.
\end{lemma}
\begin{proof}
	Let $X = U\Sigma V^*$ be the SVD of $X$. 
	Since $\norm{\cdot}$ is unitarily invariant, we have that 
	$\norm{X^*} = \norm{V\Sigma^*U^*} = \norm{\Sigma} = \norm{U\Sigma V^*} = \norm{X}$. Then, 
	\[
	\left\|\frac{X-X^*}{2i}\right\|\le \frac{\|X\|+\|X^*\|}{2}=\|X\|. \qedhere
	\]	
\end{proof}

We formulate again two Theorems, one for the closest approximation in the $2$-norm and
another one for the closest approximation in the Frobenius norm. The approximants that 
we construct are
identical, but the proofs differ significantly. Again like in the unitary case we will
change particular eigenvalues of the skew-Hermitian part to find the best approximant.

\begin{theorem}
	Let  $A\in \mathbb{C}^{n\times n}$ and  $S(A)=\frac{1}{2i}(A-A^*)$, having
        eigendecomposition $S(A)=UDU^*$. The eigenvalues are ordered: $\lambda_1\ge \cdots\ge \lambda_{k_+}> 0$, 	$\lambda_{k_++1}= \cdots= \lambda_{n-k_-}=0$  and 
$0> \lambda_{n-k_-+1} \ge \cdots \ge \lambda_{n}$, where $k_+$ stands for the number of
eigenvalues strictly greater than $0$ and $k_-$ for the eigenvalues strictly smaller than $0$.
	Then we have that
	\[
\min_{X\in \mathcal{H}_k}\|A-X\|_2=\|A-\hat A\|_2,
	\]
	where $	\hat A = A- i\, U (D-\hat D) U^*$ and $\hat D$ is a diagonal matrix with diagonal elements 
	$$
	\hat d_i=\begin{cases}
	0 & \mbox{ if } k< i\le k_+ \mbox{ or }   n-k_-< i\le n-k \\
	\lambda_i & \text{otherwise.}
	\end{cases}
	$$
		
\end{theorem}

\begin{proof}
 Theorem~\ref{thm:mainherm} implies that $\hat A=A-i U(D-\hat D) U^*$ belongs to $\mathcal{H}_k$ , since \[
S(\hat A)=\frac{\hat A-\hat A^*}{2i}=\frac{ A- A^*}{2i}- U(D-\hat D)U^*=U\hat D U^*,
\]  
has at most $k$ positive eigenvalues and at most $k$ negative eigenvalues. 

To prove that $\hat A$ is the minimizer of $\|A-X\|_2$ we have to prove that for every
$X\in \mathcal{H}_k$ we have that $\|A-X\|_2\ge \|A-\hat A\|_2= \|D-\hat
D\|_2=\max\{\lambda_{k+1}(S(A)), -\lambda_{n-k}(S(A)), 0\}$.

Assume that $X=A+\Delta$, then $S(X)=S(A)+S(\Delta)$. Using Weyl's inequality
(Theorem~\ref{thm:hj1}), we have
\[
\lambda_{k+1}(S(X))=\lambda_{k+1}(S(A)+S(\Delta))\ge \lambda_{k+1}(S(A))-\lambda_1(S(\Delta)),
\]
and therefore
\[
\|A-X\|_2=\|\Delta\|_2\ge \|S(\Delta)\|_2\ge \lambda_1(S(\Delta))\ge \lambda_{k+1}(S(A))-\lambda_{k+1}(S(X))\ge \lambda_{k+1}(S(A))
\]
since $X\in \mathcal{H}_k$ implies $\lambda_{k+1}(S(X))\le 0.$
Similarly from
\[
\lambda_{n-k}(S(A+\Delta))\le \lambda_1(S(\Delta))+\lambda_{n-k}(S(A)),
\]
we get 
\[
\|A-X\|_2\ge \lambda_1(S(\Delta)) \ge \lambda_{n-k}(S(X))-\lambda_{n-k}(S(A))\ge -\lambda_{n-k}(S(A)),
\]
since $X\in \mathcal{H}_k$ implies $\lambda_{n-k}(S(X))\ge 0.$ 
Combining the two inequalities and using the  non-negativeness of the norm,  we have
\[
\|A-X\|_2\ge \max\{\lambda_{k+1}(S(A)), -\lambda_{n-k}(S(A)),0\}.
\]

\end{proof}

We remark that, comparable to the unitary case, the minimizer in the $2$-norm is not
unique. We have constructed a solution $\hat{A}$ such that 
$\|A-\hat A\|_2= \|D-\hat
D\|_2=\max\{\lambda_{k+1}(S(A)), -\lambda_{n-k}(S(A)), 0\}$. It is, however, easy to find a
concrete example and  a
matrix $\tilde{A}$ different from $\hat{A}$ such that $\|A-\hat A\|_2= \|D-\hat D\|_2 = 
 \|D-\tilde D\|_2 = \|A-\tilde A\|_2$.


\begin{theorem}
	Let  $A\in \mathbb{C}^{n\times n}$ and  $S(A)=\frac{1}{2i}(A-A^*)$, having
        eigendecomposition $S(A)=UDU^*$. The eigenvalues are ordered: $\lambda_1\ge \cdots\ge \lambda_{k_+}> 0$, 	$\lambda_{k_++1}= \cdots= \lambda_{n-k_-}=0$  and 
$0> \lambda_{n-k_-+1} \ge \cdots \ge \lambda_{n}$, where $k_+$ stands for the number of
eigenvalues strictly greater than $0$ and $k_-$ for the eigenvalues strictly smaller than $0$.
	Then we have that
	\[
	\min_{X\in \mathcal{H}_k}\|A-X\|_F=\|A-\hat A\|_F
	\]
where	$\hat A = A- i\, U (D-\hat D) U^*,$ and
$\hat D$ is a diagonal matrix with diagonal elements 
	$$
	\hat d_i=\begin{cases}
	0 & \mbox{ if } k< i\le k_+ \mbox{ or }   n-k_-< i\le n-k \\
	\lambda_i & \text{otherwise.}
	\end{cases}
	$$
	
\end{theorem}
\begin{proof}
  We know that $\hat A\in \mathcal{H}_k$.  To prove that this is the
  minimizer, we have to show that for any $X\in\mathcal{H}_k$ we have that
  $\|A-X\|_F\geq\|A-\hat{A}\|_F$.

Consider $\Delta\in\mathbb{C}^{n\times n}$ such that $X=A + \Delta
  \in \mathcal{H}_k$. We know from Theorem~\ref{thm:mainherm} that this implies that
  $\lambda_{k+1}(S(A+\Delta ))\le 0$ and that $\lambda_{n-k}(S(A+\Delta))\ge 0$.  

Consider $\tilde{\Delta}=U^*S(\Delta)U$ and partition it as follows
	 \[
	 U^*S(A+\Delta)U=\begin{bmatrix}
	 D_1 + \tilde\Delta_{11} & \tilde\Delta_{12} & \tilde\Delta_{13}\\
	 \tilde\Delta_{21}&D_2 +\tilde\Delta_{22} & \tilde\Delta_{23}\\
	\tilde\Delta_{31}& \tilde\Delta_{32}&D_3 + \tilde\Delta_{33}
	 \end{bmatrix},
	 \]
	where $D=\diag(D_1, D_2, D_3)$ is the diagonal of the eigendecomposition
        $S(A)=UDU^*$. 
	In particular 
	$D_1 \in \mathbb{R}^{k_+ \times k_+}$ contains the eigenvalues of $S(A)$ which are
        strictly greater than $0$, and $D_3\in\mathbb{R}^{k_- \times k_-}$ contains the
        eigenvalues of $S(A)$ strictly smaller than $0$.
	
        \begin{itemize}
        \item Assume that $k_+>k$.
          The matrix $U^*S(A+\Delta)U$ is Hermitian, hence by the interlacing inequalities
          we get
          \[
          \lambda_{k+1}(D_1+\tilde\Delta_{11}) \leq \lambda_{k+1}(S(A+\Delta)) \leq 0;
          \]
          and then by Weyl's inequalities for each $k < i \leq k_+$
          \[
          \lambda_i = \lambda_i(D_1) \leq \lambda_{k+1}(D_1+ \tilde\Delta_{11}) +
          \lambda_{i-k}(\tilde\Delta_{11}) \leq \lambda_{i-k}(\tilde\Delta_{11}),
          \]
          from which we obtain $\lambda_{i-k}(\tilde\Delta_{11}) \geq \lambda_i$, and
          hence
          \begin{equation} \label{ineq1} \norm{\tilde\Delta_{11}}_F^2=\sum_{i=1}^{k_+}
            \lambda_i(\tilde\Delta_{11})^2 \geq \sum_{i=k+1}^{k_+} \lambda_i^2.
          \end{equation}
          Note that~\eqref{ineq1} holds trivially also when $k_+ \leq k$.
	
         \item Similarly, assume that $k_- > k$, and use again the interlacing inequalities to get
          \[
          \lambda_{k_- - k}(D_3 + \tilde\Delta_{33}) \geq \lambda_{n-k}(S(A+\Delta)) \geq
          0.
          \]
          Using Weyl's inequalities for each $k<i \leq k_-$ we can show that
          \[
          0 \leq \lambda_{k_- - k}(D_3 + \tilde \Delta_{33}) \leq \lambda_{k_- +1-i}(D_3)
          + \lambda_{i-k}(\tilde\Delta_{33}) = \lambda_{n+1-i} +
          \lambda_{i-k}(\tilde\Delta_{33}),
          \]
          from which we obtain $\lambda_{i-k}(\tilde\Delta_{33}) \geq - \lambda_{n+1-i}\ge
          0$, and hence
          \begin{equation} \label{ineq2} \norm{\tilde\Delta_{33}}_F^2
            \geq
            \sum_{i=k+1}^{k_-} \lambda_{i-k}(\tilde\Delta_{33})^2
            \geq \sum_{i=k+1}^{k_-} (-\lambda_{n+1-i})^2= \sum_{j=n+1-k_-}^{n-k} \lambda_{j}^2.
          \end{equation}
          We note that this equation holds trivially when $k_-\leq k$.
        \end{itemize}
	Combining the inequalities~\eqref{ineq1} and~\eqref{ineq2}, and by Lemma~\ref{lem:norm}
	stating that $
	\norm{\Delta}_F^2 \geq \norm{S(\Delta)}_F^2
	$, 
	we get
	\[
	\norm{A-X}_F^2=\norm{\Delta}_F^2 \geq \norm{\tilde\Delta_{11}}_F^2 + \norm{\tilde \Delta_{33}}_F^2 \geq \sum_{i=k+1}^{k_+} \lambda_i^2 + \sum_{j=n+1-k_-}^{n-k} \lambda_j^2,
	\]
	which is precisely $\norm{D - \hat D}_F = \norm{A -\hat A}_F$. This concludes the proof.
\end{proof}

\section{The Cayley transform}
\label{sec:cayley}

Unitary and Hermitian structures are both special cases of normal matrices. 
Even more interestingly, it is known that they can be mapped one into
the other through the use of the Cayley transform,
defined as follows:
\[
  \mathcal C(z) := \frac{z - i}{z + i}, \qquad 
  z \in \mathbb{C} \setminus \{ -i \}. 
\]
The Cayley transform is a particular case of a M\"obius transform, which permutes
projective lines of the Riemann sphere. In particular, we have that 
$\mathcal C(\mathbb R) = \mathbb S^1$. The inverse transform can be readily
expressed as 
\[
  \mathcal C^{-1}(z) = i \cdot \frac{1 + z}{1 - z}, \qquad 
  z \in \mathbb{C}\setminus \{ -1 \}. 
\]
The fact that $\mathcal C(z)$ maps Hermitian matrices into unitary ones 
has been known for a long time \cite{Ne30}. More recently, the observation
that one can switch between low-rank perturbations of these structures has 
been exploited for develop fast algorithms for unitary-plus-low-rank and 
Hermitian-plus-low-rank matrices \cite{p542,AuMaVaWa17a}. 

\begin{lemma} \label{lem:cayley}
	Let $A$ be an $n \times n$ matrix. Then we have the following.
	\begin{itemize}
		\item If $A$ does not have the eigenvalue $-i$ and $A$ is a rank 
		$k$ perturbation of a Hermitian matrix, then $\mathcal C(A)$ will be a rank $k$
		perturbation of a unitary matrix. Moreover, $\mathcal C(A)$ does not 
		possess the eigenvalue $1$. 
		\item If $A$ does not have the eigenvalue $1$ and is a rank $k$ 
		perturbation of a unitary matrix, then $\mathcal C^{-1}(A)$ will be 
		a rank $k$ perturbation of an Hermitian matrix, and 
		$C^{-1}(A)$ does not possess eigenvalue $-i$. 
	\end{itemize}
\end{lemma}

\begin{proof}
  We show that the Cayley transform (and its inverse) preserve the rank of the
  perturbation.  Note that both $\mathcal C(z)$ and $\mathcal C^{-1}(z)$ are degree
  $(1,1)$ rational functions. For a rational function $r(z)$ of degree (at most) $(d,d)$
  we know that, for any matrix $A$ and rank $k$ perturbation $E$, $r(A + E) - r(A)$ has
  rank at most $dk$.  It remains to prove that perturbation stays \emph{exactly} of rank
  $k$, and not less.  Let $A$ be Hermitian plus rank (exactly) $k$, and by contradiction,
  assume we can write $\mathcal C(A) = Q + E$, with $Q$ unitary and $\mathrm{rank}(E) = k'
  < k$.  Then, we would have that $\mathcal C^{-1}(Q + E) = A$ is a Hermitian plus rank
  $k''$ matrix where $k'' \leq k'$, leading to a contradiction.

  To prove that $\mathcal{C}(A)$ does not have $1$ in the spectrum, it suffices to note
  that $\mathcal{C}(z)$ is a bijection of the Riemann sphere, and maps the point at
  $\infty$ to $1$. Since the eigenvalues of $\mathcal C(A)$ are $\mathcal C(\lambda)$,
  with $\lambda$ the eigenvalues of $A$, we see the eigenvalue $1$ must be excluded. The
  same argument applies for the second case.
	
\end{proof}

Creating this bridge between low-rank perturbations of unitary and Hermitian
matrices enables to use the criterion that we have developed for 
detecting matrices in $\mathcal H_k$ to matrices in $\mathcal U_k$, 
and the opposite direction as well.
In fact, in the next lemma, we show that we can obtain alternative proofs for the
characterizations of $\mathcal{U}_k$ and $\mathcal{H}_k$ by simply applying the Cayley
transform.

\begin{lemma}
The Cayley transformation implies that Theorem~\ref{thm:mainortho} and
Theorem~\ref{thm:mainherm} are equivalent. 
\end{lemma}

\begin{proof}
	We start by proving
	that Theorem~\ref{thm:mainherm} implies Theorem~\ref{thm:mainortho}.
	Let $A$ be an arbitrary matrix, and assume that $1$ is not an eigenvalue.  
	We know by Lemma~\ref{lem:cayley} that $A$ is in $\mathcal U_k$
	if and only if $\mathcal C^{-1}(A) \in \mathcal H_k$. Now, $\mathcal C^{-1}(A)$
        will be a rank $k$ perturbation of a Hermitian matrix if and
	only if the Hermitian matrix $\frac{1}{2i} (\mathcal C^{-1}(A) - \mathcal C^{-1}(A)^*)$ 
	has at most $k$ positve eigenvalues and $k$ negative ones. 
	We can write 
	\[
	\frac{1}{2i} \left( \mathcal C^{-1}(A) - \mathcal C^{-1}(A)^*\right) = 
	\frac{1}{2} \cdot \left[ (A + I)^{-1} (A - I) + (A^* + I)^{-1} (A^* - I) \right], 
	\]
	Let us do a congruence
	by left-multiplying by $(A + I)$ and right multiplying by $(A^* + I)$. 
	This does not change the sign characteristic and yields
	\begin{align*}
	\frac{1}{2i} \cdot (A + I) \left[ \mathcal C^{-1}(A) - \mathcal C^{-1}(A)^* \right] 
	(A^* + I) &= \frac{1}{2} \left[ (A - I)(A^* + I) + (A + I) (A^* - I) \right]\\
	&= AA^* - I. 
	\end{align*}
	The  matrix above has eigenvalues $\lambda_i := \sigma_i^2(A) - 1$, where $\sigma_i(A)$ 
	are the singular values of $A$. Therefore, the positive eigenvalues of $\frac{1}{2i} \left(\mathcal C^{-1}(A) - \mathcal C^{-1}(A)^* \right)$  correspond to 
	singular values of $A$ larger than $1$, whereas the negative eigenvalues link to
        the singular values smaller
	than $1$. In particular, 
	$A$ is unitary plus rank $k$, if and only if the characterization
	given in Theorem~\ref{thm:mainortho} is satisfied. 
	
	Let us now consider the case
	where $A$ has $1$ as an eigenvalue. Then, we can multiply it by a unimodular scalar $\xi$ to
        get  
	$A' = \xi \cdot A$, where
	$A'$ does not have $1$ as an eigenvalue. Clearly
	$A' \in \mathcal U_k \iff A \in \mathcal U_k$, and $\sigma_i(A') = \sigma_i(A)$. Applying the previous steps to $A'$ yields the characterization
	for $A$ as well, completing the proof. 
	
	The other implication (that is, Theorem~\ref{thm:mainortho} implies Theorem~\ref{thm:mainherm}) can be obtained following the same steps backwards.
\end{proof}

\begin{remark}
	We emphasize that, although in principle the Cayley transform enables
	to study unitary matrices looking at Hermitian ones (and the other
	way around), it cannot be used to answer questions about the
	closest unitary or Hermitian matrix. In fact, this transformation
	does not preserve the distances. 
\end{remark}

\section{Construction of the representations}
\label{sec:lanczos}

We present in this section a proof-of-concept algorithm to show how one can use the results in this paper to construct, given a matrix $A \in \mathcal{H}_k$ (resp. $\mathcal{U}_k$), matrices $G,B \in \mathbb{C}^{n\times \ell}$, with $\ell \leq k$, such that $A-GB^*$ is Hermitian (resp. unitary). The procedure identifies the minimum possible rank $\ell$ automatically given the matrix $A$ only, and requires $\mathcal O(n^2\ell)$ flops for a full matrix.

\subsection{Constructing a representation $A = H + GB^*$}

If $A = H + GB^*$, with $H = H^*$, the
Hermitian matrix $S(A) = \frac{1}{2i} (A - A^*)$ has rank $k_+ + k_- \leq 2k$ (where $k_+, k_-$ are as in the proof of Theorem~\ref{thm:mainherm}), hence the Lanczos algorithm (with a random starting vector $b$) applied to $S(A)$ will break down after at most $k_+ + k_-$ steps, in exact arithmetic, giving an approximation $S(A) = \frac{1}{2i} (A - A^*) \approx W T W^*$.

We apply to $T$ the procedure described in the proof of Theorem~\ref{thm:mainherm}, which
is fully constructive, to recover a decomposition $
  T \approx  \hat{B}\hat{C}^* +\hat{C}\hat{B}^* 
$, neglecting the eigenvalues smaller than a prescribed truncation threshold. 
This implies that we can write $\frac{1}{2i} (A - A^*) \approx B C^* + C B^*$, where
$C := W\hat{C}, B := W\hat{B}$. Then, we construct the final decomposition by setting
$H = A - 2i C B^*$. The procedure is sketched in the pseudocode of Algorithm~\ref{alg:lanczos}. 

\begin{algorithm}
\caption{Lanczos-based scheme to recover the Hermitian-plus-low-rank decomposition. A truncation threshold $\varepsilon$ is given.}\label{alg:lanczos}
\begin{algorithmic}[1]
	\Procedure{hk\_find}{$A$, $\varepsilon$} 
	\State $S \gets \frac{1}{2i} (A - A^*)$
	\State $W, T \gets \Call{Lanczos}{S, \varepsilon}$ \Comment{$S = WTW^* + \mathcal O(\varepsilon)$}
	\State $\hat{B}, \hat{C} \gets \Call{ComputeFactors}{T, \varepsilon}$
	  \Comment{$T = \hat{B}\hat{C}^* + \hat{C}\hat{B}^* + \mathcal{O}(\varepsilon)$, using Lemma~\ref{lem2x2herm}}
	\State $B, C \gets W \hat{B}, W \hat{C}$
	\State $G \gets 2iC$
	\State $H \gets A - GB^*$
	\State \Return $\frac{1}{2} (H + H^*), G, B$. 
	\EndProcedure
\end{algorithmic}
\end{algorithm}

Note that in the unlikely case in which the process terminates early, $WTW^*$ is not equal
to $S$, but only to its restriction on the maximal Krylov subspace. Hence, when we detect
that $WTW^*-S$ is still too large, we can continue the Lanczos iteration but with a randomly chosen vector.

\begin{remark}
	A reconstruction procedure can be obtained from any method to approximate the
        range of $S$ in $O(n^2k)$ flops.
        Indeed, once one obtains an orthogonal basis $W$ for $\operatorname{Im} S$, one
        can compute $W^*SW = T$ and continue as above.
\end{remark}

\subsection{Constructing a representation $A = Q + GB^*$}

The case of a unitary-plus-rank-$k$ matrix can be solved with the same ideas, relying on the Golub--Kahan bidiagonalization.

If $A\in\mathcal{U}_k$, then Theorem~\ref{thm:mainortho} shows that the matrices $A^*A$ and $AA^*$ have (at most) $m: = k_+ + k_- + 1$ distinct eigenvalues: $k_+$ greater than 1, $k_-$ smaller than 1, and the eigenvalue $1$, possibly with high multiplicity. Hence, the Lanczos process on each of them breaks down after at most $m$ steps. Indeed, if $v_1,v_2,\dots,v_n$ is an eigenvector basis for $A^*A$, ordered so that $v_m,v_{m+1},\dots,v_n$ are eigenvectors with eigenvalue $1$, then we can write the initial vector for the Lanczos process as $b = \alpha_1 v_1 + \dots + \alpha_{m-1}v_{m-1} + w$ where $w = \alpha_m v_m + \alpha_{m+1}v_{m+1}+ \dots + \alpha_n v_n$; the vector $w$ is also an eigenvector $w$ with eigenvalue $1$. Thus $b$ belongs to the invariant subspace $\operatorname{span}(v_1,v_2,\dots,v_{m-1},w)$, which is also (generically) its maximal Krylov subspace.

It is well established that the Golub-Kahan bidiagonalization is equivalent to running the Lanczos process to $A^*A$ and $AA^*$, hence it will also break down after $m$ steps (generically, when there is no earlier breakdown), returning matrices $U_1, M, V_1$ such that the decomposition
\[
A \approx \begin{bmatrix}
    U_1 & U_2
\end{bmatrix}
\begin{bmatrix}
    M \\
    &I
\end{bmatrix}
\begin{bmatrix}
    V_1 & V_2
\end{bmatrix}^*
\]
holds, with $M\in\mathbb{C}^{m\times m}$ upper bidiagonal, and $\begin{bmatrix}
    U_1 & U_2
\end{bmatrix}$ and $\begin{bmatrix}
    V_1 & V_2
\end{bmatrix}$ orthogonal.

We can use the constructive argument in the proof of Theorem~\ref{thm:mainortho} to decompose $M \approx Q + \hat{G}\hat{B}^*$, neglecting the singular values such that $|\sigma_i-1|$ is below a certain truncation threshold. Then, $G = U_1\hat{G}$, $B = V_1\hat{B}$ give the required decomposition. The procedure is sketched in the pseudocode of Algorithm~\ref{alg:golubkahan}.
\begin{algorithm}
	\caption{Golub--Kahan-based scheme to recover the unitary-plus-low-rank decomposition. A truncation threshold $\varepsilon$ is given.}\label{alg:golubkahan}
	\begin{algorithmic}[1]
		\Procedure{uk\_find}{$A$, $\varepsilon$} 
		\State $U_1, M, V_1 \gets \Call{GolubKahan}{A, \varepsilon}$ \Comment{$A = [U_1,U_2]\operatorname{diag}(M,I)[V_1,V_2]^* + \mathcal \mathcal{O}(\varepsilon)$}
		\State $\hat{G}, \hat{B} \gets \Call{ComputeFactors}{M, \varepsilon}$ 	\Comment{$M = Q + \hat{G}\hat{B}^* + \mathcal{O}(\varepsilon)$ using Lemma~\ref{lem:2x2}}
		\State $G, B \gets U_1 \hat{G}, V_1\hat{B}$
		\State $Q \gets A - GB^*$
		\State \Return $Q, G, B$. 
		\EndProcedure
	\end{algorithmic}
\end{algorithm}

\section{Numerical experiments}
\label{sec:experiments}

\subsection{Accuracy of the reconstruction procedure} \label{sec:accuracy}

We have implemented Algorithm~\ref{alg:lanczos} and Algorithm~\ref{alg:golubkahan}
in Matlab, and ran some tests with randomly generated matrices 
to validate the procedures. The algorithms appear to be quite robust 
and succeeded in all our
tests in retrieving a decomposition with the correct rank and a small error.

In each test, we generated a random $n\times n$ Hermitian or unitary plus rank-$k$ matrix. For the Hermitian case, this is achieved with the Matlab commands
\begin{verbatim}
rng('default');
H = randn(n, n) + 1i * randn(n, n);
H = H + H';
[U,~] = qr(randn(n, k) + 1i * randn(n, k));
[V,~] = qr(randn(n, k) + 1i * randn(n, k));
A = H + U * diag(sv) * V';
\end{verbatim}
Here \verb!sv! are logarithmically distributed singular values between $1$ and a parameter $\sigma$. In the unitary case, the matrix $H$ is replaced by the $Q$ factor
of a QR factorization of a random $n \times n$ matrix with entries distributed
as $N(0, 1)$, i.e., \verb|[Q,~] = qr(randn(n))|.

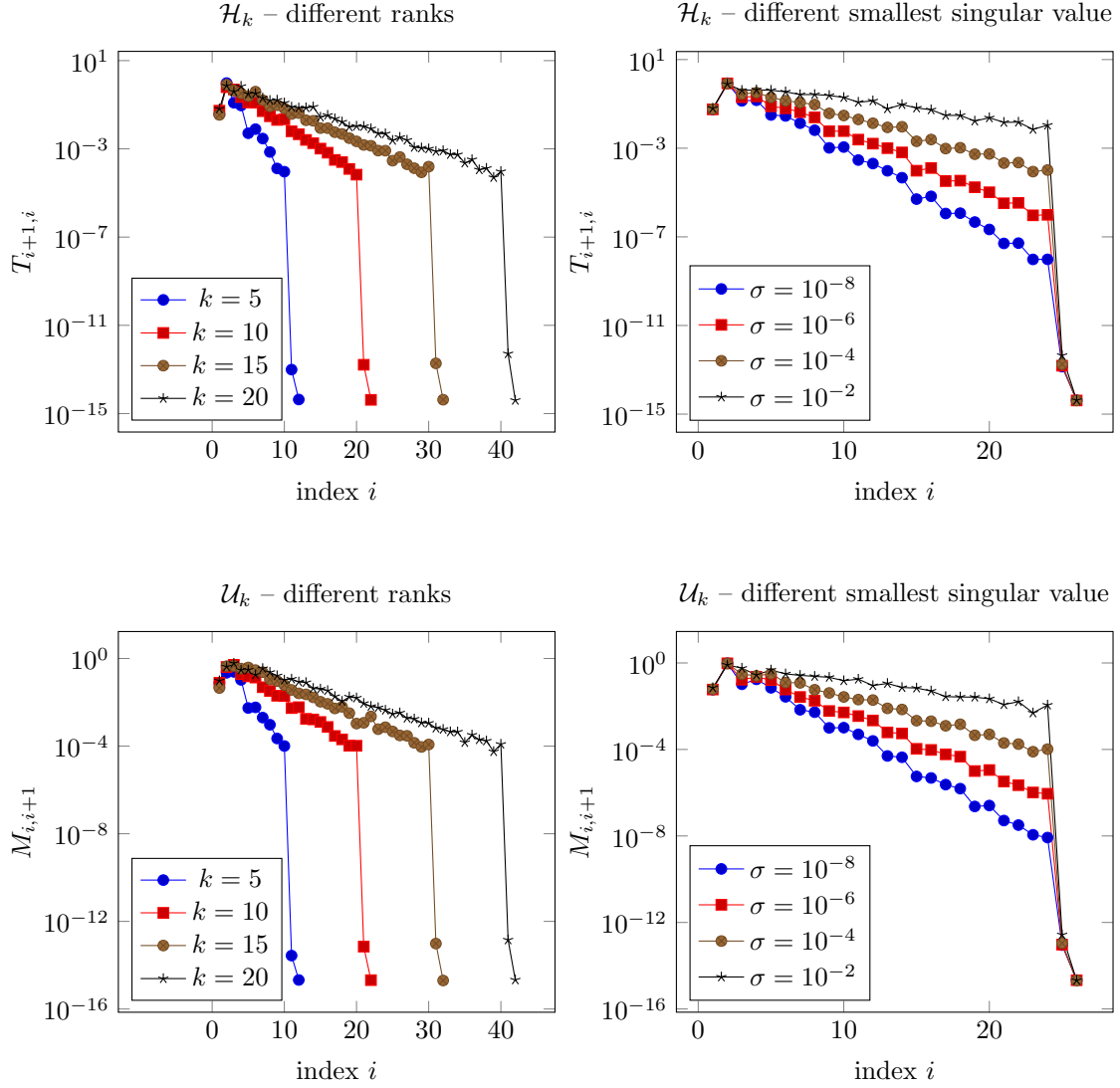
\begin{figure}
	\begin{tikzpicture}
	\begin{semilogyaxis}[width=.5\textwidth, height=0.32\textheight, xlabel={index $i$}, ylabel={$T_{i+1,i}$}, legend pos=south west, xmin =-13, xtick = {0,10,20,30,40}, title={$\mathcal H_k$ -- different 	ranks}]
	
	\addplot table [x=idx,y=beta1] {hk_k.dat};
	\addplot table [x=idx,y=beta2] {hk_k.dat};
	\addplot table [x=idx,y=beta3] {hk_k.dat};
	\addplot table [x=idx,y=beta4] {hk_k.dat};
	
	\legend{$k = 5$,
		$k = 10$,
		$k = 15$,
		$k = 20$}
	
	\end{semilogyaxis}
	\end{tikzpicture}~\begin{tikzpicture}
	\begin{semilogyaxis}[width=.5\textwidth, height=0.32\textheight, xlabel={index $i$}, ylabel={$T_{i+1,i}$}, 
	legend pos=south west,title={$\mathcal H_k$ -- different
		smallest singular value}]
	
	\addplot table [x=idx,y=beta1] {hk_sigma.dat};
	\addplot table [x=idx,y=beta2] {hk_sigma.dat};
	\addplot table [x=idx,y=beta3] {hk_sigma.dat};
	\addplot table [x=idx,y=beta4] {hk_sigma.dat};
	
	\legend{$\sigma = 10^{-8}$,
		$\sigma = 10^{-6}$,
		$\sigma = 10^{-4}$,
		$\sigma = 10^{-2}$}
	
	\end{semilogyaxis}
	\end{tikzpicture} \\[24pt]
	\begin{tikzpicture}
	\begin{semilogyaxis}[width=.5\textwidth, height=0.32\textheight, xlabel={index $i$}, ylabel={$M_{i,i+1}$}, legend pos=south west, xmin =-13, xtick = {0,10,20,30,40}, title={$\mathcal U_k$ -- different 	ranks}]
	
	\addplot table [x=idx,y=beta1] {uk_k.dat};
	\addplot table [x=idx,y=beta2] {uk_k.dat};
	\addplot table [x=idx,y=beta3] {uk_k.dat};
	\addplot table [x=idx,y=beta4] {uk_k.dat};
	
	\legend{$k = 5$,
		$k = 10$,
		$k = 15$,
		$k = 20$}
	\end{semilogyaxis}
	\end{tikzpicture}~\begin{tikzpicture}
	\begin{semilogyaxis}[width=.5\textwidth, height=0.32\textheight, xlabel={index $i$}, ylabel={$M_{i,i+1}$}, legend pos=south west, title={$\mathcal U_k$ -- different
		smallest singular value}]
	
	\addplot table [x=idx,y=beta1] {uk_sigma.dat};
	\addplot table [x=idx,y=beta2] {uk_sigma.dat};
	\addplot table [x=idx,y=beta3] {uk_sigma.dat};
	\addplot table [x=idx,y=beta4] {uk_sigma.dat};
	
	\legend{$\sigma = 10^{-8}$,
		$\sigma = 10^{-6}$,
		$\sigma = 10^{-4}$,
		$\sigma = 10^{-2}$}
	\end{semilogyaxis}
	\end{tikzpicture}
	\caption{(Top) Magnitude of the subdiagonal entries $T_{i+1,i}$ (in the
		Lanczos process for the Hermitian case), or (bottom) the superdiagonal ones $M_{i,i+1}$
		(in the Golub-Kahan bidiagonalization scheme), computed in
		the recovery of the approximate factors $H,G,B$ in a 
		Hermitian plus-low-rank matrix $A = H + GB^*$, or $Q,G,B$ of 
		the unitary plus-low-rank matrix $A = Q + GB^*$. The
		tests have been performed for different smallest singular values $\sigma$ of $GB^*$
		and different ranks $k$. 
	} 
	\label{fig:hkbeta}
\end{figure}

We ran experiments with varying values of $n$, $k$, and $\sigma$; we report
the convergence history by plotting the values of the subdiagonal
entries obtained in the Lanczos process (in the Hermitian case), 
or in the Golub-Kahan bidiagonalization (for the unitary case). Experiments revealed that the parameter
$n$ has no visible effect on the convergence or the accuracy, 
so we only plotted the behavior
with different values of $k$ and $\sigma$, which are reported in Figure~\ref{fig:hkbeta}. 
The graphs show that there
is a sharp drop in their magnitude after $2k$ steps, which is what is
expected since $\operatorname{rank}(S)=2k$, generically. 
The
magnitude of the subdiagonal entries $T_{i+1,i}$ 
or superdiagonal ones $M_{i,i+1}$ reflects the decay
in the singular values in the rank correction. 

For the Hermitian case, 
the relative residual $\norm{\frac12(H+H^*)+GB^*-A}_2 / {\|A\|_2}$ has 
been verified to be around $6\cdot 10^{-17}$ in all the tests. 
Note that there is a difference in Algorithm~\ref{alg:lanczos} and 
Algorithm~\ref{alg:golubkahan}: in the former, the matrix $H$ is guaranteed to be Hermitian,
since it is symmetrized explicitly at the end of the algorithm, so it makes sense to
measure the reconstruction error as $\norm{\frac12(H+H^*)+GB^*-A}_2 / {\|A\|_2}$. In the
latter, the unitary factor $Q$ is obtained by the difference $Q = A - GB^*$: hence, we expect $\norm{Q + GB^* - A}_2 / \norm{A}_2$ to be of the order of the machine
precision (the error is given just by the subtraction), but $Q$ will only be
approximately unitary. Therefore, in this case we measure the error as 
$\max_{j} |\sigma_j(Q) - 1|$, which is the distance of $Q$ to a unitary matrix
in the Euclidean norm. 
In the experiments reported in Figure~\ref{fig:hkbeta}, the errors are enclosed in 
$[3u, 4u]$, where $u$ is the machine precision $u \approx 2.22 \cdot 10^{-16}$.

\subsection{Rank structure in linearizations}

A classical example of matrices with Hermitian or unitary plus low-rank structure
are linearizations of polynomials and matrix polynomials. The
most well-known linearization is probably the classical Frobenius companion matrix, 
obtained in MATLAB by the command \texttt{compan(p)}, where \texttt{p} is a vector
with the coefficients of a polynomial.
For many linearizations, one can find, by direct inspection, a value $k$ such that
they are $\mathcal{H}_k$ or $\mathcal{U}_k$; for instance the Frobenius companion form 
is seen to be $\mathcal{U}_1$, as one can turn it into a cyclic shift matrix by changing only
the top row.
In this section we use some of these examples, to test the recovery procedure,
and to confirm that the computed values $k$  are optimal.

\subsubsection{The Fiedler pentadiagonal linearization}

A classical variant of the scalar companion matrix is the pentadiagonal Fiedler matrix
obtained by permuting the elementary Fiedler factors according to the odd-even permutation.  Given a monic polynomial of degree $n$, $p(x)=x^{n}-\sum_{i=0}^{n-1} p_i x^i$, define
$$
F_0:=\left[\begin{array}{cc}p_0& \\
&I_{n-1}\end{array}\right],
$$
$$
F_i:=\left[\begin{array}{ccc}I_{i-1} && \\
&G(p_i)& \\
&&I_{n-i-1}\end{array}\right],\quad G(p_i):=\left[\begin{array}{cc}0&1 \\
1&p_i\end{array}\right], \quad i=1, \ldots, n-1.
$$
Then, the matrix $F=F_1F_3\cdots F_{2\lfloor \frac{n}{2} \rfloor -1}\, F_0F_2 \cdots F_{2\lfloor\frac{n-1}{2}\rfloor}$ is a linearization of $p(x)$ and has the pentadiagonal structure depicted in Figure~\ref{fig:spy_fiedler}. From~\cite{delcorso2018factoring,de2013condition} we know that these matrices are unitary-plus-rank-$k$ with $k$ at most $\lceil \frac{n}{2}\rceil$.  In particular one can observe that 
\begin{equation}\label{uni_sparse}
F=Q+GB^*,
\end{equation}  
where $Q$ is an orthogonal matrix (depicted in red in Figure~\ref{fig:spy_fiedler})  whose nonzeros are precisely in position $(2; 1)$, $(n-1 ; n)$ 
and those of the form $(2j-1; 2j + 1) $ and $(2j; 2j + 2)$ for $ j\ge1$; it is the matrix constructed analogously to $F$ but starting from the polynomial $p(x)=x^n-1$. Indeed, the nonzero entries of $F-Q$ (depicted in blue in Figure~\ref{fig:spy_fiedler}, plus an additional one in $(2;1)$) appear only in its odd-indexed columns, plus the last one; hence the rank of this correction is clearly bounded by $\lceil \frac{n}{2}\rceil$.

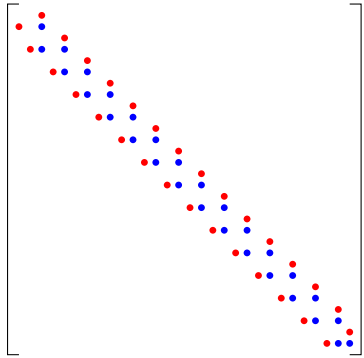
\begin{figure}
\begin{center}
\begin{tikzpicture}[scale=0.15,rotate=270]
	\def\csz{.3}
	\def\fiedlersz{30}
	\pgfmathparse{\fiedlersz-2}
	\foreach \j in {1, 3, ..., \pgfmathresult} {
		\fill[red] (\j,\j+2) circle (\csz);
		\fill[red] (\j+3,\j+1) circle (\csz);
	}
	\fill[red] (2,1) circle (\csz);
	\pgfmathparse{\fiedlersz-3}
	\foreach \j in {1, 3, ..., \pgfmathresult} {
	\fill[blue] (\j+1,\j+2) circle (\csz);
	\fill[blue] (\j+3,\j+2) circle (\csz);
	}
	\pgfmathparse{\fiedlersz-1}
	\fill[blue] (\fiedlersz,\fiedlersz) circle (\csz);
	\fill[red] (\pgfmathresult,\fiedlersz) circle (\csz);
	\pgfmathsetmacro{\fiedlerszp}{\fiedlersz+1}
	\draw (0,1) -- (0,0) -- (\fiedlerszp,0) -- (\fiedlerszp,1);
	\draw (0,\fiedlersz) -- (0,\fiedlerszp) -- 
	    (\fiedlerszp,\fiedlerszp) -- (\fiedlerszp,\fiedlersz);
\end{tikzpicture}
\end{center}
\caption{Nonzero pattern of a Fiedler pentadiagonal matrix of size $30$. The entries
  displayed in red are all $1$ except for $F_{2,1}=p_0$. In particular, when $p(x)=x^n-1$, the matrix $F$ is a (unitary) permutation matrix that has $\pm 1$ in the entries in red and $0$ elsewhere.} \label{fig:spy_fiedler}

\end{figure}
Applying Algorithm~\ref{alg:golubkahan} to a pentadiagonal Fiedler of size $n=512$ generated from a monic random scalar polynomial we obtain the subdiagonal entries $T_{i+1, i}$ plotted in Figure~\ref{fig:pentafiedler}. As expected we get that the first subdiagonal entry below the threshold $\epsilon=1.0e-14$ is $T_{513, 512}$, so that $m=512$ and $k_+=k_-=256$. The algorithm computes correctly a representation $F=\hat{Q}+\hat{G}\hat{B}^*$ with $\hat{G}, \hat{B}\in \mathbb{C}^{512\times 256}$. Thus the experiments reveal that the bound $k \leq \lceil \frac{n}{2} \rceil$ is tight. However, the computed unitary term $\hat{Q}$ does not coincide with the one obtained theoretically in~\eqref{uni_sparse}. This fact should not be surprising, since the representation is not necessarily unique.

\pgfplotstableread{
	idx beta
	   1.0000e+00   1.7148e+00
	2.0000e+00   2.1349e+00
	3.0000e+00   2.7653e+00
	4.0000e+00   1.9604e+00
	5.0000e+00   2.0559e+00
	6.0000e+00   1.7828e+00
	7.0000e+00   1.7070e+00
	8.0000e+00   1.6834e+00
	9.0000e+00   1.5829e+00
	1.0000e+01   1.7627e+00
	1.1000e+01   1.7912e+00
	1.2000e+01   1.6986e+00
	1.3000e+01   1.6818e+00
	1.4000e+01   1.6364e+00
	1.5000e+01   1.6394e+00
	1.6000e+01   1.7058e+00
	1.7000e+01   1.8079e+00
	1.8000e+01   1.5661e+00
	1.9000e+01   1.4864e+00
	2.0000e+01   1.7358e+00
	2.1000e+01   1.5178e+00
	2.2000e+01   1.5969e+00
	2.3000e+01   1.8932e+00
	2.4000e+01   1.6109e+00
	2.5000e+01   1.6016e+00
	2.6000e+01   1.4360e+00
	2.7000e+01   1.6185e+00
	2.8000e+01   1.8300e+00
	2.9000e+01   1.5959e+00
	3.0000e+01   1.5510e+00
	3.1000e+01   1.8817e+00
	3.2000e+01   1.4965e+00
	3.3000e+01   1.4736e+00
	3.4000e+01   1.5990e+00
	3.5000e+01   1.7947e+00
	3.6000e+01   1.7001e+00
	3.7000e+01   1.6689e+00
	3.8000e+01   1.3793e+00
	3.9000e+01   1.4706e+00
	4.0000e+01   1.4657e+00
	4.1000e+01   1.5816e+00
	4.2000e+01   1.7169e+00
	4.3000e+01   1.4972e+00
	4.4000e+01   1.5848e+00
	4.5000e+01   1.4857e+00
	4.6000e+01   1.3206e+00
	4.7000e+01   1.1982e+00
	4.8000e+01   1.3294e+00
	4.9000e+01   1.2304e+00
	5.0000e+01   1.3329e+00
	5.1000e+01   1.1989e+00
	5.2000e+01   1.2036e+00
	5.3000e+01   1.3261e+00
	5.4000e+01   1.2797e+00
	5.5000e+01   1.4357e+00
	5.6000e+01   1.2007e+00
	5.7000e+01   1.2529e+00
	5.8000e+01   1.2174e+00
	5.9000e+01   1.1579e+00
	6.0000e+01   1.2215e+00
	6.1000e+01   1.2312e+00
	6.2000e+01   1.2574e+00
	6.3000e+01   1.2561e+00
	6.4000e+01   1.2538e+00
	6.5000e+01   1.1972e+00
	6.6000e+01   1.2726e+00
	6.7000e+01   1.2245e+00
	6.8000e+01   1.3434e+00
	6.9000e+01   1.2185e+00
	7.0000e+01   1.2552e+00
	7.1000e+01   1.1520e+00
	7.2000e+01   1.1123e+00
	7.3000e+01   1.0914e+00
	7.4000e+01   1.3040e+00
	7.5000e+01   1.2312e+00
	7.6000e+01   1.1812e+00
	7.7000e+01   1.2354e+00
	7.8000e+01   1.2933e+00
	7.9000e+01   1.2645e+00
	8.0000e+01   1.0905e+00
	8.1000e+01   1.1386e+00
	8.2000e+01   1.1703e+00
	8.3000e+01   1.3325e+00
	8.4000e+01   1.2079e+00
	8.5000e+01   1.1414e+00
	8.6000e+01   1.1696e+00
	8.7000e+01   1.1099e+00
	8.8000e+01   1.1179e+00
	8.9000e+01   9.8846e-01
	9.0000e+01   1.0959e+00
	9.1000e+01   1.0693e+00
	9.2000e+01   1.1374e+00
	9.3000e+01   1.2029e+00
	9.4000e+01   1.2638e+00
	9.5000e+01   1.1276e+00
	9.6000e+01   1.1140e+00
	9.7000e+01   1.0809e+00
	9.8000e+01   1.0093e+00
	9.9000e+01   1.1438e+00
	1.0000e+02   1.0403e+00
	1.0100e+02   1.0560e+00
	1.0200e+02   1.0526e+00
	1.0300e+02   1.0416e+00
	1.0400e+02   1.0029e+00
	1.0500e+02   9.5410e-01
	1.0600e+02   1.0198e+00
	1.0700e+02   1.0134e+00
	1.0800e+02   9.7543e-01
	1.0900e+02   1.0702e+00
	1.1000e+02   9.8705e-01
	1.1100e+02   9.8321e-01
	1.1200e+02   1.0960e+00
	1.1300e+02   1.0184e+00
	1.1400e+02   1.0971e+00
	1.1500e+02   1.0935e+00
	1.1600e+02   9.7206e-01
	1.1700e+02   9.5257e-01
	1.1800e+02   1.0140e+00
	1.1900e+02   9.8049e-01
	1.2000e+02   1.0215e+00
	1.2100e+02   9.2829e-01
	1.2200e+02   1.0829e+00
	1.2300e+02   8.8885e-01
	1.2400e+02   1.0603e+00
	1.2500e+02   9.4774e-01
	1.2600e+02   9.8162e-01
	1.2700e+02   1.0794e+00
	1.2800e+02   9.5858e-01
	1.2900e+02   9.9020e-01
	1.3000e+02   1.0124e+00
	1.3100e+02   8.6157e-01
	1.3200e+02   1.0847e+00
	1.3300e+02   1.0273e+00
	1.3400e+02   8.6190e-01
	1.3500e+02   1.0271e+00
	1.3600e+02   9.8540e-01
	1.3700e+02   1.0284e+00
	1.3800e+02   8.8453e-01
	1.3900e+02   9.5761e-01
	1.4000e+02   9.4574e-01
	1.4100e+02   1.0232e+00
	1.4200e+02   1.0138e+00
	1.4300e+02   1.0421e+00
	1.4400e+02   9.5858e-01
	1.4500e+02   9.3912e-01
	1.4600e+02   1.0332e+00
	1.4700e+02   8.8396e-01
	1.4800e+02   1.1253e+00
	1.4900e+02   9.2079e-01
	1.5000e+02   1.0654e+00
	1.5100e+02   1.0003e+00
	1.5200e+02   9.9961e-01
	1.5300e+02   1.0006e+00
	1.5400e+02   9.9266e-01
	1.5500e+02   9.0850e-01
	1.5600e+02   1.0251e+00
	1.5700e+02   9.4212e-01
	1.5800e+02   9.6457e-01
	1.5900e+02   8.6799e-01
	1.6000e+02   9.5225e-01
	1.6100e+02   1.0328e+00
	1.6200e+02   9.9720e-01
	1.6300e+02   9.3860e-01
	1.6400e+02   1.0329e+00
	1.6500e+02   9.9063e-01
	1.6600e+02   9.7119e-01
	1.6700e+02   1.0803e+00
	1.6800e+02   9.0888e-01
	1.6900e+02   9.2724e-01
	1.7000e+02   1.0727e+00
	1.7100e+02   8.9580e-01
	1.7200e+02   9.4735e-01
	1.7300e+02   9.2391e-01
	1.7400e+02   9.3941e-01
	1.7500e+02   1.0188e+00
	1.7600e+02   9.0250e-01
	1.7700e+02   8.4552e-01
	1.7800e+02   1.0439e+00
	1.7900e+02   9.8254e-01
	1.8000e+02   8.6846e-01
	1.8100e+02   8.8750e-01
	1.8200e+02   1.0817e+00
	1.8300e+02   8.7985e-01
	1.8400e+02   9.4593e-01
	1.8500e+02   9.3154e-01
	1.8600e+02   8.7575e-01
	1.8700e+02   8.1670e-01
	1.8800e+02   9.0110e-01
	1.8900e+02   8.3177e-01
	1.9000e+02   9.6754e-01
	1.9100e+02   1.0423e+00
	1.9200e+02   8.6270e-01
	1.9300e+02   9.0079e-01
	1.9400e+02   1.1134e+00
	1.9500e+02   7.8362e-01
	1.9600e+02   8.9980e-01
	1.9700e+02   1.0587e+00
	1.9800e+02   9.2178e-01
	1.9900e+02   9.0120e-01
	2.0000e+02   9.6755e-01
	2.0100e+02   1.0297e+00
	2.0200e+02   9.4004e-01
	2.0300e+02   9.7544e-01
	2.0400e+02   1.0695e+00
	2.0500e+02   8.5263e-01
	2.0600e+02   9.6232e-01
	2.0700e+02   7.8814e-01
	2.0800e+02   9.6675e-01
	2.0900e+02   8.6924e-01
	2.1000e+02   9.5570e-01
	2.1100e+02   7.4275e-01
	2.1200e+02   6.6856e-01
	2.1300e+02   8.7464e-01
	2.1400e+02   9.4281e-01
	2.1500e+02   9.1526e-01
	2.1600e+02   6.5755e-01
	2.1700e+02   6.0021e-01
	2.1800e+02   8.3721e-01
	2.1900e+02   7.3324e-01
	2.2000e+02   8.4552e-01
	2.2100e+02   8.0906e-01
	2.2200e+02   7.0123e-01
	2.2300e+02   8.8566e-01
	2.2400e+02   7.2140e-01
	2.2500e+02   7.2572e-01
	2.2600e+02   5.8623e-01
	2.2700e+02   7.0497e-01
	2.2800e+02   6.1279e-01
	2.2900e+02   7.4351e-01
	2.3000e+02   7.7298e-01
	2.3100e+02   6.8829e-01
	2.3200e+02   6.4394e-01
	2.3300e+02   6.2911e-01
	2.3400e+02   6.5707e-01
	2.3500e+02   7.3713e-01
	2.3600e+02   6.7077e-01
	2.3700e+02   6.6770e-01
	2.3800e+02   6.1788e-01
	2.3900e+02   6.1972e-01
	2.4000e+02   7.3468e-01
	2.4100e+02   6.7246e-01
	2.4200e+02   6.5880e-01
	2.4300e+02   6.4636e-01
	2.4400e+02   6.3741e-01
	2.4500e+02   6.6910e-01
	2.4600e+02   6.7899e-01
	2.4700e+02   6.7384e-01
	2.4800e+02   7.5111e-01
	2.4900e+02   6.2836e-01
	2.5000e+02   5.4817e-01
	2.5100e+02   5.5990e-01
	2.5200e+02   6.1684e-01
	2.5300e+02   6.5664e-01
	2.5400e+02   5.9543e-01
	2.5500e+02   6.0806e-01
	2.5600e+02   5.9997e-01
	2.5700e+02   7.2960e-01
	2.5800e+02   7.0302e-01
	2.5900e+02   5.4696e-01
	2.6000e+02   5.9816e-01
	2.6100e+02   6.6441e-01
	2.6200e+02   6.0416e-01
	2.6300e+02   5.9109e-01
	2.6400e+02   6.1051e-01
	2.6500e+02   5.5979e-01
	2.6600e+02   5.9966e-01
	2.6700e+02   5.3115e-01
	2.6800e+02   4.9552e-01
	2.6900e+02   5.8823e-01
	2.7000e+02   5.8758e-01
	2.7100e+02   5.7849e-01
	2.7200e+02   5.6807e-01
	2.7300e+02   6.6225e-01
	2.7400e+02   6.3846e-01
	2.7500e+02   5.6555e-01
	2.7600e+02   5.4551e-01
	2.7700e+02   5.5081e-01
	2.7800e+02   5.8394e-01
	2.7900e+02   5.5150e-01
	2.8000e+02   5.5106e-01
	2.8100e+02   5.8557e-01
	2.8200e+02   6.1988e-01
	2.8300e+02   5.3401e-01
	2.8400e+02   5.4476e-01
	2.8500e+02   6.5971e-01
	2.8600e+02   5.1582e-01
	2.8700e+02   6.4636e-01
	2.8800e+02   5.6741e-01
	2.8900e+02   5.1568e-01
	2.9000e+02   4.8332e-01
	2.9100e+02   5.6679e-01
	2.9200e+02   5.6090e-01
	2.9300e+02   4.8792e-01
	2.9400e+02   6.0571e-01
	2.9500e+02   5.3680e-01
	2.9600e+02   4.8987e-01
	2.9700e+02   4.7453e-01
	2.9800e+02   5.0977e-01
	2.9900e+02   5.0958e-01
	3.0000e+02   4.4990e-01
	3.0100e+02   5.1315e-01
	3.0200e+02   4.9786e-01
	3.0300e+02   5.3806e-01
	3.0400e+02   5.7645e-01
	3.0500e+02   5.1716e-01
	3.0600e+02   5.4080e-01
	3.0700e+02   5.0721e-01
	3.0800e+02   5.8121e-01
	3.0900e+02   5.1167e-01
	3.1000e+02   4.8855e-01
	3.1100e+02   6.0272e-01
	3.1200e+02   4.8246e-01
	3.1300e+02   5.0960e-01
	3.1400e+02   4.7132e-01
	3.1500e+02   4.7396e-01
	3.1600e+02   5.6016e-01
	3.1700e+02   4.3468e-01
	3.1800e+02   4.6745e-01
	3.1900e+02   4.8436e-01
	3.2000e+02   4.5786e-01
	3.2100e+02   5.3775e-01
	3.2200e+02   4.8755e-01
	3.2300e+02   4.8777e-01
	3.2400e+02   4.6022e-01
	3.2500e+02   4.7444e-01
	3.2600e+02   5.6266e-01
	3.2700e+02   4.2189e-01
	3.2800e+02   4.0361e-01
	3.2900e+02   4.5472e-01
	3.3000e+02   4.6922e-01
	3.3100e+02   4.0402e-01
	3.3200e+02   5.3405e-01
	3.3300e+02   5.0119e-01
	3.3400e+02   4.2083e-01
	3.3500e+02   4.3177e-01
	3.3600e+02   4.8281e-01
	3.3700e+02   4.2398e-01
	3.3800e+02   4.2604e-01
	3.3900e+02   4.1384e-01
	3.4000e+02   4.2996e-01
	3.4100e+02   3.9925e-01
	3.4200e+02   3.2655e-01
	3.4300e+02   3.8532e-01
	3.4400e+02   4.3529e-01
	3.4500e+02   4.8842e-01
	3.4600e+02   4.4086e-01
	3.4700e+02   3.6872e-01
	3.4800e+02   4.1839e-01
	3.4900e+02   3.6128e-01
	3.5000e+02   4.3327e-01
	3.5100e+02   4.7028e-01
	3.5200e+02   3.4094e-01
	3.5300e+02   4.2037e-01
	3.5400e+02   3.9151e-01
	3.5500e+02   4.3758e-01
	3.5600e+02   3.7412e-01
	3.5700e+02   3.6883e-01
	3.5800e+02   3.8544e-01
	3.5900e+02   3.4052e-01
	3.6000e+02   3.3698e-01
	3.6100e+02   3.3251e-01
	3.6200e+02   3.7210e-01
	3.6300e+02   4.1050e-01
	3.6400e+02   3.4558e-01
	3.6500e+02   3.6794e-01
	3.6600e+02   3.5350e-01
	3.6700e+02   3.7032e-01
	3.6800e+02   3.6543e-01
	3.6900e+02   3.1566e-01
	3.7000e+02   3.6334e-01
	3.7100e+02   3.0898e-01
	3.7200e+02   3.5501e-01
	3.7300e+02   3.8432e-01
	3.7400e+02   2.8683e-01
	3.7500e+02   3.6595e-01
	3.7600e+02   3.7129e-01
	3.7700e+02   2.4216e-01
	3.7800e+02   3.9230e-01
	3.7900e+02   3.3999e-01
	3.8000e+02   2.7999e-01
	3.8100e+02   3.9148e-01
	3.8200e+02   3.3619e-01
	3.8300e+02   3.1029e-01
	3.8400e+02   2.9326e-01
	3.8500e+02   3.3576e-01
	3.8600e+02   3.4688e-01
	3.8700e+02   2.8799e-01
	3.8800e+02   3.5718e-01
	3.8900e+02   3.0378e-01
	3.9000e+02   3.1981e-01
	3.9100e+02   3.1392e-01
	3.9200e+02   2.8336e-01
	3.9300e+02   3.5430e-01
	3.9400e+02   3.0716e-01
	3.9500e+02   3.0267e-01
	3.9600e+02   4.1038e-01
	3.9700e+02   3.1218e-01
	3.9800e+02   3.0242e-01
	3.9900e+02   3.1489e-01
	4.0000e+02   2.4742e-01
	4.0100e+02   3.0799e-01
	4.0200e+02   2.9457e-01
	4.0300e+02   2.2774e-01
	4.0400e+02   3.0001e-01
	4.0500e+02   2.8176e-01
	4.0600e+02   2.7335e-01
	4.0700e+02   3.0601e-01
	4.0800e+02   3.1446e-01
	4.0900e+02   2.6060e-01
	4.1000e+02   2.4436e-01
	4.1100e+02   2.7273e-01
	4.1200e+02   3.8707e-01
	4.1300e+02   2.3491e-01
	4.1400e+02   3.0317e-01
	4.1500e+02   2.7112e-01
	4.1600e+02   2.5919e-01
	4.1700e+02   3.2125e-01
	4.1800e+02   2.9669e-01
	4.1900e+02   2.6380e-01
	4.2000e+02   2.3263e-01
	4.2100e+02   3.3917e-01
	4.2200e+02   2.7646e-01
	4.2300e+02   2.9806e-01
	4.2400e+02   2.6762e-01
	4.2500e+02   3.0897e-01
	4.2600e+02   2.8463e-01
	4.2700e+02   2.6487e-01
	4.2800e+02   2.5381e-01
	4.2900e+02   4.1439e-01
	4.3000e+02   2.3189e-01
	4.3100e+02   2.2051e-01
	4.3200e+02   3.5576e-01
	4.3300e+02   2.6685e-01
	4.3400e+02   3.4123e-01
	4.3500e+02   2.7840e-01
	4.3600e+02   1.8680e-01
	4.3700e+02   2.7253e-01
	4.3800e+02   3.1787e-01
	4.3900e+02   2.2047e-01
	4.4000e+02   1.7430e-01
	4.4100e+02   3.4219e-01
	4.4200e+02   3.4745e-01
	4.4300e+02   1.8851e-01
	4.4400e+02   2.2754e-01
	4.4500e+02   3.7947e-01
	4.4600e+02   2.4372e-01
	4.4700e+02   1.2957e-01
	4.4800e+02   2.7764e-01
	4.4900e+02   3.6889e-01
	4.5000e+02   2.0025e-01
	4.5100e+02   2.1969e-01
	4.5200e+02   2.0839e-01
	4.5300e+02   3.8503e-01
	4.5400e+02   1.7227e-01
	4.5500e+02   1.8347e-01
	4.5600e+02   1.7354e-01
	4.5700e+02   2.6748e-01
	4.5800e+02   3.0713e-01
	4.5900e+02   2.0999e-01
	4.6000e+02   1.6545e-01
	4.6100e+02   2.1192e-01
	4.6200e+02   3.4574e-01
	4.6300e+02   2.0640e-01
	4.6400e+02   1.6622e-01
	4.6500e+02   2.0603e-01
	4.6600e+02   2.1888e-01
	4.6700e+02   2.5906e-01
	4.6800e+02   1.6995e-01
	4.6900e+02   1.8962e-01
	4.7000e+02   1.4664e-01
	4.7100e+02   2.9694e-01
	4.7200e+02   2.6616e-01
	4.7300e+02   1.2589e-01
	4.7400e+02   1.8770e-01
	4.7500e+02   2.8949e-01
	4.7600e+02   3.5840e-01
	4.7700e+02   1.3309e-01
	4.7800e+02   1.7797e-01
	4.7900e+02   1.2448e-01
	4.8000e+02   3.5938e-01
	4.8100e+02   2.8246e-01
	4.8200e+02   6.9364e-02
	4.8300e+02   1.9897e-01
	4.8400e+02   2.7922e-01
	4.8500e+02   3.0860e-01
	4.8600e+02   1.3956e-01
	4.8700e+02   1.5865e-01
	4.8800e+02   1.1690e-01
	4.8900e+02   3.6048e-01
	4.9000e+02   2.1364e-01
	4.9100e+02   5.4719e-02
	4.9200e+02   7.9071e-02
	4.9300e+02   3.1849e-01
	4.9400e+02   1.8570e-01
	4.9500e+02   1.1246e-01
	4.9600e+02   8.4609e-02
	4.9700e+02   4.5928e-01
	4.9800e+02   3.9751e-02
	4.9900e+02   2.0705e-01
	5.0000e+02   4.3650e-02
	5.0100e+02   2.9434e-01
	5.0200e+02   3.9770e-02
	5.0300e+02   1.4944e-02
	5.0400e+02   1.5066e-01
	5.0500e+02   9.3061e-02
	5.0600e+02   8.8438e-03
	5.0700e+02   1.4174e-01
	5.0800e+02   4.8269e-02
	5.0900e+02   2.5858e-03
	5.1000e+02   9.8405e-06
	5.1100e+02   4.4370e-12
	5.1200e+02   2.2296e-31
}\loadedtable
\begin{figure}
\begin{tikzpicture}
\begin{semilogyaxis}[width=\textwidth, height=0.4\textheight, xlabel={index $i$}, ylabel={$M_{i,i+1}$}, legend pos=north east]

\addplot+[mark size=0.7pt] table[x=idx,y=beta] \loadedtable;

\end{semilogyaxis}
\end{tikzpicture}
\caption{Superdiagonal entries obtained from applying the Golub-Kahan bidiagonalization scheme to  the pentadiagonal Fiedler matrix of a random polynomial with $n=513$.} \label{fig:pentafiedler}
\end{figure}
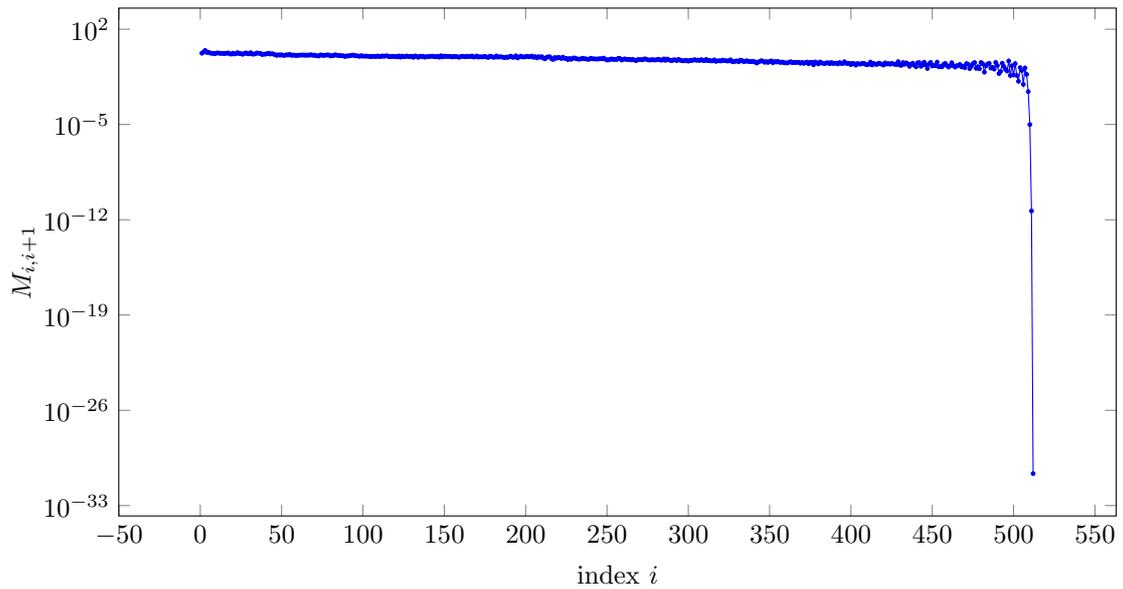

\subsubsection{The colleague linearization}
Consider an $m \times m$ matrix polynomial $P(\lambda)$ expressed in the Chebyshev basis
\[
  P(\lambda) = P_0 T_0(\lambda) + P_1 T_1(\lambda) + \ldots + P_d T_d(\lambda), \qquad 
  T_{j+1}(\lambda) = 2\lambda T_j(\lambda) - T_{j-1}(\lambda), 
\]
and $T_0(\lambda) = 1, T_1(\lambda) = \lambda$. Such a polynomial can be linearized 
using the \emph{colleague matrix} 
\[
C=
\begin{bmatrix}
    P_1 & P_2 & P_3 & \cdots & P_d\\
    \frac12I &  & \frac12I \\
    & \ddots & & \ddots \\
    & & \frac12I & & \frac12I\\
    & & & I
\end{bmatrix} \in \mathbb{R}^{md\times md},
\]
partitioned into blocks of size $m\times m$ each. It is simple to see that this matrix belongs to $\mathcal{H}_{2m}$, as it is sufficient to modify the first and last block row
to obtain the Hermitian matrix $\operatorname{tridiag}(\frac12,0,\frac12) \otimes I_m$. However, $D^{-1}CD \in \mathcal{H}_m$ for a suitable diagonal scaling matrix $D$, so one may wonder if the rank $2m$ of this correction  is optimal or if it can be lowered with a more clever construction.


We apply Algorithm~\ref{alg:lanczos} to a large-scale (sparse) matrix $C$, with $d=m=100$, in which the first block row contains random coefficients (obtained with \texttt{randn(m)}). We plot the subdiagonal entries $T_{i+1,i}$ obtained in Figure~\ref{fig:chebyshev}.
\pgfplotstableread{
	idx beta
   1.0000e+00   1.4573e+01
   2.0000e+00   1.0020e+02
   3.0000e+00   1.7214e+01
   4.0000e+00   1.0055e+02
   5.0000e+00   1.5532e+01
   6.0000e+00   9.6828e+01
   7.0000e+00   2.4465e+01
   8.0000e+00   7.1062e+01
   9.0000e+00   7.1892e+01
   1.0000e+01   2.3788e+01
   1.1000e+01   9.7222e+01
   1.2000e+01   1.7025e+01
   1.3000e+01   9.9472e+01
   1.4000e+01   2.3768e+01
   1.5000e+01   6.8516e+01
   1.6000e+01   7.5711e+01
   1.7000e+01   2.1509e+01
   1.8000e+01   9.9436e+01
   1.9000e+01   1.5953e+01
   2.0000e+01   1.0037e+02
   2.1000e+01   1.6502e+01
   2.2000e+01   9.9801e+01
   2.3000e+01   1.5550e+01
   2.4000e+01   9.8006e+01
   2.5000e+01   2.2910e+01
   2.6000e+01   6.5984e+01
   2.7000e+01   7.6830e+01
   2.8000e+01   2.0675e+01
   2.9000e+01   9.9568e+01
   3.0000e+01   1.7090e+01
   3.1000e+01   9.7152e+01
   3.2000e+01   2.6554e+01
   3.3000e+01   5.7604e+01
   3.4000e+01   8.2990e+01
   3.5000e+01   1.8103e+01
   3.6000e+01   9.8009e+01
   3.7000e+01   1.5984e+01
   3.8000e+01   1.0038e+02
   3.9000e+01   1.6137e+01
   4.0000e+01   1.0136e+02
   4.1000e+01   1.6175e+01
   4.2000e+01   9.7012e+01
   4.3000e+01   2.7758e+01
   4.4000e+01   5.3055e+01
   4.5000e+01   8.6679e+01
   4.6000e+01   1.7932e+01
   4.7000e+01   9.9133e+01
   4.8000e+01   1.6959e+01
   4.9000e+01   9.5243e+01
   5.0000e+01   3.5375e+01
   5.1000e+01   4.4141e+01
   5.2000e+01   9.1299e+01
   5.3000e+01   1.7085e+01
   5.4000e+01   9.7528e+01
   5.5000e+01   1.4665e+01
   5.6000e+01   9.9435e+01
   5.7000e+01   1.5288e+01
   5.8000e+01   9.8425e+01
   5.9000e+01   1.5257e+01
   6.0000e+01   9.2902e+01
   6.1000e+01   4.2665e+01
   6.2000e+01   3.1389e+01
   6.3000e+01   9.5728e+01
   6.4000e+01   1.4443e+01
   6.5000e+01   9.8779e+01
   6.6000e+01   1.9818e+01
   6.7000e+01   8.0410e+01
   6.8000e+01   6.2009e+01
   6.9000e+01   2.4592e+01
   7.0000e+01   9.7167e+01
   7.1000e+01   1.5073e+01
   7.2000e+01   1.0051e+02
   7.3000e+01   1.3658e+01
   7.4000e+01   9.9697e+01
   7.5000e+01   1.3440e+01
   7.6000e+01   9.5720e+01
   7.7000e+01   2.1136e+01
   7.8000e+01   5.9493e+01
   7.9000e+01   8.1572e+01
   8.0000e+01   1.7897e+01
   8.1000e+01   9.9163e+01
   8.2000e+01   1.3207e+01
   8.3000e+01   9.2338e+01
   8.4000e+01   3.7501e+01
   8.5000e+01   3.3327e+01
   8.6000e+01   9.4189e+01
   8.7000e+01   1.5449e+01
   8.8000e+01   9.8448e+01
   8.9000e+01   1.2820e+01
   9.0000e+01   9.8379e+01
   9.1000e+01   1.2723e+01
   9.2000e+01   9.8493e+01
   9.3000e+01   1.7464e+01
   9.4000e+01   6.9842e+01
   9.5000e+01   7.2270e+01
   9.6000e+01   1.6741e+01
   9.7000e+01   9.8954e+01
   9.8000e+01   1.2839e+01
   9.9000e+01   9.4159e+01
   1.0000e+02   3.5121e+01
   1.0100e+02   3.4771e+01
   1.0200e+02   9.3473e+01
   1.0300e+02   1.4002e+01
   1.0400e+02   9.9897e+01
   1.0500e+02   1.1689e+01
   1.0600e+02   1.0022e+02
   1.0700e+02   1.2929e+01
   1.0800e+02   9.7725e+01
   1.0900e+02   1.7130e+01
   1.1000e+02   5.9591e+01
   1.1100e+02   7.9215e+01
   1.1200e+02   1.7424e+01
   1.1300e+02   9.7596e+01
   1.1400e+02   1.2741e+01
   1.1500e+02   9.1948e+01
   1.1600e+02   3.7748e+01
   1.1700e+02   2.8745e+01
   1.1800e+02   9.5212e+01
   1.1900e+02   1.2574e+01
   1.2000e+02   9.9413e+01
   1.2100e+02   1.1680e+01
   1.2200e+02   9.7683e+01
   1.2300e+02   1.2334e+01
   1.2400e+02   9.7774e+01
   1.2500e+02   2.1364e+01
   1.2600e+02   4.8526e+01
   1.2700e+02   8.5782e+01
   1.2800e+02   1.3486e+01
   1.2900e+02   9.8454e+01
   1.3000e+02   1.2581e+01
   1.3100e+02   7.9545e+01
   1.3200e+02   5.9116e+01
   1.3300e+02   1.8859e+01
   1.3400e+02   9.8502e+01
   1.3500e+02   1.0761e+01
   1.3600e+02   9.8381e+01
   1.3700e+02   1.1806e+01
   1.3800e+02   9.8172e+01
   1.3900e+02   1.9946e+01
   1.4000e+02   5.3277e+01
   1.4100e+02   8.3116e+01
   1.4200e+02   1.1763e+01
   1.4300e+02   9.7770e+01
   1.4400e+02   1.0798e+01
   1.4500e+02   9.8031e+01
   1.4600e+02   2.1484e+01
   1.4700e+02   4.5922e+01
   1.4800e+02   8.7637e+01
   1.4900e+02   1.2106e+01
   1.5000e+02   9.8388e+01
   1.5100e+02   1.2718e+01
   1.5200e+02   7.2965e+01
   1.5300e+02   6.6108e+01
   1.5400e+02   1.2180e+01
   1.5500e+02   9.7930e+01
   1.5600e+02   9.3716e+00
   1.5700e+02   9.9291e+01
   1.5800e+02   9.0227e+00
   1.5900e+02   9.7686e+01
   1.6000e+02   1.4553e+01
   1.6100e+02   6.0717e+01
   1.6200e+02   7.8630e+01
   1.6300e+02   1.2966e+01
   1.6400e+02   9.6660e+01
   1.6500e+02   1.0632e+01
   1.6600e+02   8.1186e+01
   1.6700e+02   5.6916e+01
   1.6800e+02   1.2591e+01
   1.6900e+02   9.8630e+01
   1.7000e+02   9.1815e+00
   1.7100e+02   9.8602e+01
   1.7200e+02   7.7768e+00
   1.7300e+02   9.7530e+01
   1.7400e+02   1.4620e+01
   1.7500e+02   6.3157e+01
   1.7600e+02   7.6567e+01
   1.7700e+02   9.4106e+00
   1.7800e+02   9.8238e+01
   1.7900e+02   1.1569e+01
   1.8000e+02   6.9882e+01
   1.8100e+02   6.9459e+01
   1.8200e+02   1.0541e+01
   1.8300e+02   9.7252e+01
   1.8400e+02   8.9321e+00
   1.8500e+02   9.7735e+01
   1.8600e+02   7.8358e+00
   1.8700e+02   9.5700e+01
   1.8800e+02   1.9462e+01
   1.8900e+02   3.7299e+01
   1.9000e+02   9.0354e+01
   1.9100e+02   7.0818e+00
   1.9200e+02   9.6624e+01
   1.9300e+02   1.9595e+01
   1.9400e+02   3.6215e+01
   1.9500e+02   9.0917e+01
   1.9600e+02   5.7729e+00
   1.9700e+02   9.8677e+01
   1.9800e+02   7.6004e+00
   1.9900e+02   9.7947e+01
   2.0000e+02   8.7574e+00
   2.0100e+02   7.9441e+01
   2.0200e+02   6.0104e+01
   2.0300e+02   7.0235e+00
   2.0400e+02   9.7699e+01
   2.0500e+02   1.0050e+01
   2.0600e+02   5.0222e+01
   2.0700e+02   8.3535e+01
   2.0800e+02   5.5972e+00
   2.0900e+02   9.9229e+01
   2.1000e+02   8.2225e+00
   2.1100e+02   9.8833e+01
   2.1200e+02   5.8993e+00
   2.1300e+02   6.9966e+01
   2.1400e+02   6.8581e+01
   2.1500e+02   6.3722e+00
   2.1600e+02   9.7237e+01
   2.1700e+02   1.1196e+01
   2.1800e+02   3.2478e+01
   2.1900e+02   9.3657e+01
   2.2000e+02   6.2106e+00
   2.2100e+02   9.8256e+01
   2.2200e+02   2.8022e+00
   2.2300e+02   9.7059e+01
   2.2400e+02   1.4583e+01
   2.2500e+02   1.8741e+01
   2.2600e+02   9.4892e+01
   2.2700e+02   2.9438e+00
   2.2800e+02   4.7728e+01
   2.2900e+02   8.5010e+01
   2.3000e+02   1.5104e+00
   2.3100e+02   9.8666e+01
   2.3200e+02   1.1936e+00
   2.3300e+02   4.8237e-01
   2.3400e+02   8.6395e-02
   2.3500e+02   5.0667e-02
   2.3600e+02   4.8994e-01
   2.3700e+02   7.7502e-03
   2.3800e+02   4.9172e-01
   2.3900e+02   8.2790e-03
   2.4000e+02   4.9134e-01
   2.4100e+02   8.6330e-03
   2.4200e+02   4.9001e-01
   2.4300e+02   8.9065e-03
   2.4400e+02   4.9249e-01
   2.4500e+02   8.8411e-03
   2.4600e+02   4.9111e-01
   2.4700e+02   8.2136e-03
   2.4800e+02   4.9200e-01
   2.4900e+02   8.3311e-03
   2.5000e+02   4.9223e-01
   2.5100e+02   8.0956e-03
   2.5200e+02   4.8038e-01
   2.5300e+02   1.0558e-01
   2.5400e+02   3.5497e-02
   2.5500e+02   4.9064e-01
   2.5600e+02   7.6561e-03
   2.5700e+02   4.9244e-01
   2.5800e+02   8.2807e-03
   2.5900e+02   4.9214e-01
   2.6000e+02   7.1349e-03
   2.6100e+02   4.9307e-01
   2.6200e+02   6.8521e-03
   2.6300e+02   4.9292e-01
   2.6400e+02   7.2741e-03
   2.6500e+02   4.9346e-01
   2.6600e+02   6.3365e-03
   2.6700e+02   4.9277e-01
   2.6800e+02   7.5347e-03
   2.6900e+02   4.9271e-01
   2.7000e+02   7.5996e-03
   2.7100e+02   4.9056e-01
   2.7200e+02   4.2922e-02
   2.7300e+02   7.0239e-02
   2.7400e+02   4.8885e-01
   2.7500e+02   7.1254e-03
   2.7600e+02   4.9339e-01
   2.7700e+02   6.4205e-03
   2.7800e+02   4.9350e-01
   2.7900e+02   6.4108e-03
   2.8000e+02   4.9355e-01
   2.8100e+02   6.9500e-03
   2.8200e+02   4.9368e-01
   2.8300e+02   6.2280e-03
   2.8400e+02   4.9305e-01
   2.8500e+02   5.9023e-03
   2.8600e+02   4.9435e-01
   2.8700e+02   6.1153e-03
   2.8800e+02   4.9316e-01
   2.8900e+02   6.3805e-03
   2.9000e+02   4.9309e-01
   2.9100e+02   3.8454e-02
   2.9200e+02   8.2567e-02
   2.9300e+02   4.8566e-01
   2.9400e+02   6.8875e-03
   2.9500e+02   4.9384e-01
   2.9600e+02   5.2474e-03
   2.9700e+02   4.9405e-01
   2.9800e+02   5.4988e-03
   2.9900e+02   4.9492e-01
   3.0000e+02   4.9965e-03
   3.0100e+02   4.9552e-01
   3.0200e+02   4.6956e-03
   3.0300e+02   4.9460e-01
   3.0400e+02   6.2399e-03
   3.0500e+02   4.9443e-01
   3.0600e+02   5.5994e-03
   3.0700e+02   4.9355e-01
   3.0800e+02   6.7988e-03
   3.0900e+02   4.8550e-01
   3.1000e+02   8.9692e-02
   3.1100e+02   3.1845e-02
   3.1200e+02   4.9348e-01
   3.1300e+02   4.7846e-03
   3.1400e+02   4.9565e-01
   3.1500e+02   4.7040e-03
   3.1600e+02   4.9547e-01
   3.1700e+02   4.1900e-03
   3.1800e+02   4.9549e-01
   3.1900e+02   3.8329e-03
   3.2000e+02   4.9675e-01
   3.2100e+02   3.4883e-03
   3.2200e+02   4.9607e-01
   3.2300e+02   4.5904e-03
   3.2400e+02   4.9480e-01
   3.2500e+02   4.8536e-03
   3.2600e+02   4.9460e-01
   3.2700e+02   2.0741e-02
   3.2800e+02   1.2748e-01
   3.2900e+02   4.7891e-01
   3.3000e+02   4.3099e-03
   3.3100e+02   4.9511e-01
   3.3200e+02   4.7137e-03
   3.3300e+02   4.9577e-01
   3.3400e+02   3.6792e-03
   3.3500e+02   4.9728e-01
   3.3600e+02   3.0472e-03
   3.3700e+02   4.9642e-01
   3.3800e+02   2.9520e-03
   3.3900e+02   4.9561e-01
   3.4000e+02   4.6837e-03
   3.4100e+02   4.9607e-01
   3.4200e+02   4.1472e-03
   3.4300e+02   4.9612e-01
   3.4400e+02   1.3084e-02
   3.4500e+02   1.2944e-01
   3.4600e+02   4.7892e-01
   3.4700e+02   3.5855e-03
   3.4800e+02   4.9716e-01
   3.4900e+02   3.0618e-03
   3.5000e+02   4.9662e-01
   3.5100e+02   2.7477e-03
   3.5200e+02   4.9747e-01
   3.5300e+02   2.2224e-03
   3.5400e+02   4.9758e-01
   3.5500e+02   2.7511e-03
   3.5600e+02   4.9758e-01
   3.5700e+02   2.6030e-03
   3.5800e+02   4.9751e-01
   3.5900e+02   2.3484e-03
   3.6000e+02   4.7887e-01
   3.6100e+02   1.3603e-01
   3.6200e+02   7.1474e-03
   3.6300e+02   4.9694e-01
   3.6400e+02   3.0390e-03
   3.6500e+02   4.9753e-01
   3.6600e+02   2.3100e-03
   3.6700e+02   4.9807e-01
   3.6800e+02   2.1488e-03
   3.6900e+02   4.9789e-01
   3.7000e+02   1.5757e-03
   3.7100e+02   4.9827e-01
   3.7200e+02   2.0936e-03
   3.7300e+02   4.9822e-01
   3.7400e+02   1.9199e-03
   3.7500e+02   4.9498e-01
   3.7600e+02   5.3435e-02
   3.7700e+02   1.4863e-02
   3.7800e+02   4.9782e-01
   3.7900e+02   2.0884e-03
   3.8000e+02   4.9829e-01
   3.8100e+02   1.5233e-03
   3.8200e+02   4.9853e-01
   3.8300e+02   1.4489e-03
   3.8400e+02   4.9849e-01
   3.8500e+02   1.2080e-03
   3.8600e+02   4.9876e-01
   3.8700e+02   1.2489e-03
   3.8800e+02   4.9901e-01
   3.8900e+02   1.3635e-03
   3.9000e+02   4.9077e-01
   3.9100e+02   8.8609e-02
   3.9200e+02   4.1658e-03
   3.9300e+02   4.9938e-01
   3.9400e+02   7.1844e-04
   3.9500e+02   4.9884e-01
   3.9600e+02   7.8529e-04
   3.9700e+02   4.9955e-01
   3.9800e+02   6.0061e-04
   3.9900e+02   4.9884e-01
   4.0000e+02   9.6488e-04
   4.0100e+02   4.9952e-01
   4.0200e+02   3.4459e-04
   4.0300e+02   4.9460e-01
   4.0400e+02   7.0649e-02
   4.0500e+02   2.3069e-03
   4.0600e+02   4.9958e-01
   4.0700e+02   2.5382e-04
   4.0800e+02   4.9974e-01
   4.0900e+02   2.5951e-04
   4.1000e+02   4.9987e-01
   4.1100e+02   2.3629e-05
   4.1200e+02   4.9998e-01
   4.1300e+02   1.6935e-05
   4.1400e+02   4.7230e-23
}\loadedtable
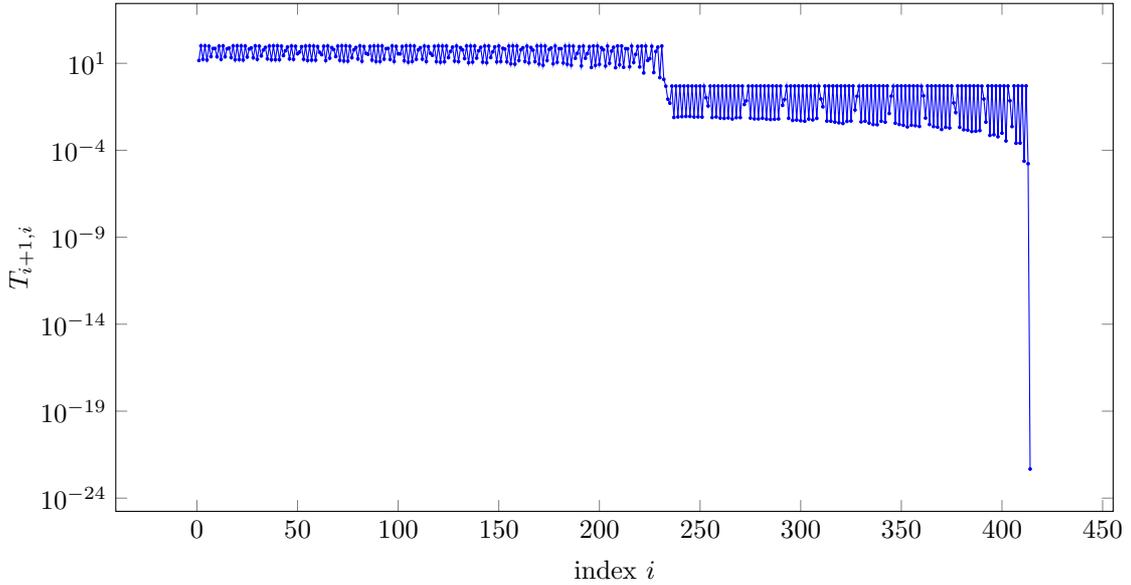
\begin{figure}
\begin{tikzpicture}
\begin{semilogyaxis}[width=\textwidth, height=0.4\textheight, xlabel={index $i$}, ylabel={$T_{i+1,i}$}, legend pos=north east]

\addplot+[mark size=0.5pt] table[x=idx,y=beta] \loadedtable;

\end{semilogyaxis}
\end{tikzpicture}
\caption{Subdiagonal  entries obtained for the colleague matrix of a random polynomial with $d=100,m=100$.} \label{fig:chebyshev}
\end{figure}
The first negligible subdiagonal entry is $T_{415,414}$, giving an invariant subspace of dimension $414$. Continuing the computation on $T$ reveals only $400$ nonzero eigenvalues, and returns a representation $C=A+UV\in\mathcal{H}_{200}$. Hence our experiment confirms that in this example the minimum rank of the correction is $2m=200$.

\subsection{Structure loss in computing the Schur form}

Given a companion matrix
$C$, in its Schur form $C=UTU^*$ the upper triangular  factor
$T$ is unitary-plus-rank-1, in exact arithmetic (since $C$ is so). As
described in the introduction, several numerical methods in the
literature try to exploit this structure, using special
representations to enforce exactly the unitary plus rank 1 structure.
If instead an approximation $\tilde{T}$ is computed using the standard
QR algorithm (Matlab's \texttt{schur(C)}), can we measure the loss of
structure in $\tilde{T}$, i.e., the distance between $\tilde{T}$ and
the closest matrix which is unitary-plus-rank-1?

We have run some experiments in which this distance is computed using the formula in Theorem~\ref{thm:distanceUk}, in two different cases:
\begin{itemize}
	\item The companion matrix of a polynomial whose roots are random numbers generated from a normal distribution with mean $0$ and variance $1$, i.e., Matlab's \texttt{compan(poly(randn(n, 1)))};
	\item The companion matrix of Wilkinson's polynomial, i.e., the polynomial with roots $1,2,\dots,n$.
\end{itemize}
The singular values of $\tilde{T}$ have been computed using extended precision arithmetic, to  get a more accurate result.

Figure~\ref{tbl:reldist} displays the (relative) distance from structure
\[
\frac{\norm{\tilde{T}-X}_2}{\norm{\tilde{T}}_2}, \quad X = \arg\min_{X \in \mathcal{U}_k} \norm{\tilde{T}-X}_2
\]
for  different matrix sizes $n$. This distance is always within a moderate multiple of the machine precision, which is to be expected because the Schur form is computed with a backward stable algorithm. It appears that the loss of structure is less pronounced for the Wilkinson polynomial; note, though, that $\norm{T}=\norm{C}$ is much larger (for $n=60$, $\norm{C}\approx 2\times 10^{83}$ for the Wilkinson polynomial vs.\ $\norm{C}\approx 3\times 10^{7}$ for the random polynomial).

Another interesting quantity is
\[
\frac{\norm{\tilde{T}-X}_2}{\norm{\tilde{T}-T}_2}, \quad X = \arg\min_{X \in \mathcal{U}_k} \norm{\tilde{T}-X}_2,
\]
i.e., the relative amount (measured as a fraction in $[0,1]$) of the total error on $\tilde{T}$ that can be attributed to the loss of structure. If this ratio is close to $0$, then it means that the main effect of the error is perturbing $T$ to another matrix \emph{inside} $\mathcal{U}_k$, while if it is close to $1$ then its main effect is perturbing it to a matrix \emph{outside} $\mathcal{U}_k$, and hence it would make sense to consider a projection procedure to map it back to $\mathcal{U}_k$.

We display this quantity in Figure~\ref{tbl:percent}. It is again smaller for the Wilkinson polynomial, and in both cases it seems to decrease slowly as the dimension $n$ increases.

\pgfplotstableread{
   n            reldist      reldistwilk  percent      percentwilk           
   4.0000e+00   5.8093e-16   2.7097e-16   4.7192e-01   2.8062e-01
   5.0000e+00   1.0323e-15   6.6203e-17   4.3350e-01   8.5813e-02
   6.0000e+00   3.7414e-16   1.6910e-17   9.9164e-02   4.2785e-02
   7.0000e+00   4.5583e-16   1.4195e-17   5.5818e-01   1.5958e-02
   8.0000e+00   4.0709e-16   2.0332e-17   3.7952e-01   1.2418e-02
   9.0000e+00   7.1055e-16   2.8192e-17   3.8956e-01   2.1434e-02
   1.0000e+01   3.1031e-16   4.5569e-17   1.3248e-01   5.8663e-02
   1.1000e+01   1.5711e-16   3.1074e-17   1.3752e-01   1.6914e-02
   1.2000e+01   2.4116e-16   9.6917e-18   1.3426e-01   1.3083e-02
   1.3000e+01   2.1500e-16   4.0197e-18   1.2237e-01   3.7613e-03
   1.4000e+01   1.8628e-16   1.3333e-17   1.4050e-01   8.0583e-03
   1.5000e+01   3.3065e-16   1.1237e-17   1.6648e-01   2.6594e-02
   1.6000e+01   4.8153e-16   2.1038e-17   1.0266e-01   8.9716e-03
   1.7000e+01   6.4314e-16   1.7142e-18   2.5603e-01   9.8197e-04
   1.8000e+01   3.5744e-16   1.1998e-17   8.0388e-02   4.6312e-03
   1.9000e+01   4.5118e-16   8.7340e-18   1.7134e-01   1.1376e-02
   2.0000e+01   2.8333e-16   1.2393e-17   2.0555e-01   3.1995e-03
   2.1000e+01   2.4100e-16   1.6978e-17   1.2820e-01   4.2806e-03
   2.2000e+01   2.8953e-16   4.3996e-17   1.3790e-01   2.4229e-02
   2.3000e+01   3.5149e-16   2.2403e-17   7.1473e-02   5.2873e-03
   2.4000e+01   2.8430e-16   5.4494e-18   1.6125e-01   6.5849e-03
   2.5000e+01   3.7628e-16   9.0800e-18   2.2740e-01   1.0670e-02
   2.6000e+01   1.9804e-16   1.1001e-17   1.3972e-01   1.8646e-02
   2.7000e+01   2.8134e-16   7.5054e-18   2.0841e-01   2.6674e-03
   2.8000e+01   3.5190e-16   2.4637e-17   1.3383e-01   9.9825e-03
   2.9000e+01   1.8523e-16   2.3556e-18   6.2543e-02   7.7741e-03
   3.0000e+01   2.3086e-16   9.5546e-18   4.6256e-02   2.7327e-03
   3.1000e+01   1.4260e-16   8.4196e-18   4.5918e-02   2.9413e-03
   3.2000e+01   1.5123e-16   3.4090e-18   9.8698e-02   2.6921e-03
   3.3000e+01   3.4171e-16   2.5426e-17   7.5770e-02   5.4324e-02
   3.4000e+01   1.9933e-16   2.4900e-17   9.6797e-02   1.2094e-02
   3.5000e+01   3.5556e-16   4.2968e-18   1.1022e-01   3.4685e-03
   3.6000e+01   2.9101e-16   3.7319e-18   8.2137e-02   1.7626e-03
   3.7000e+01   2.4436e-16   1.4170e-17   9.0056e-02   2.7263e-02
   3.8000e+01   2.1501e-16   1.0257e-17   1.2883e-01   7.6325e-03
   3.9000e+01   2.3615e-16   3.2687e-18   1.0200e-01   4.6984e-03
   4.0000e+01   1.1507e-16   8.8254e-18   8.0105e-02   7.9120e-03
   4.1000e+01   4.3842e-16   1.1581e-17   9.6308e-02   1.2826e-02
   4.2000e+01   3.6628e-16   9.6106e-18   6.5900e-02   1.1459e-02
   4.3000e+01   2.5107e-16   1.6646e-17   6.0673e-02   1.5962e-02
   4.4000e+01   2.7645e-16   2.6493e-18   6.5049e-02   8.4270e-04
   4.5000e+01   3.5288e-16   7.6078e-18   1.3038e-01   2.2309e-03
   4.6000e+01   3.7403e-16   2.1856e-17   5.5821e-02   1.6596e-02
   4.7000e+01   1.8070e-16   6.3429e-18   2.6321e-02   4.8094e-03
   4.8000e+01   2.9035e-16   1.5393e-18   8.3171e-02   9.2064e-04
   4.9000e+01   1.6712e-16   2.7956e-18   6.4492e-02   2.3032e-03
   5.0000e+01   1.8070e-16   1.1563e-16   1.5734e-01   6.0967e-02
   5.1000e+01   3.8625e-16   1.0043e-17   1.2514e-01   6.8034e-03
   5.2000e+01   1.2701e-16   6.1527e-18   2.0642e-02   1.4891e-03
   5.3000e+01   4.4801e-16   1.0238e-18   5.9433e-02   1.2917e-03
   5.4000e+01   3.1681e-16   9.5079e-19   4.2011e-02   2.9976e-04
   5.5000e+01   8.5301e-17   1.2032e-17   7.9995e-02   5.4309e-03
   5.6000e+01   3.2872e-16   4.1284e-18   8.4807e-02   6.3114e-03
   5.7000e+01   2.7951e-16   5.7482e-18   7.0772e-02   2.0444e-03
   5.8000e+01   3.3955e-16   4.1615e-18   7.2882e-02   1.4316e-03
   5.9000e+01   1.0740e-16   5.4511e-18   4.9937e-02   4.8842e-03
   6.0000e+01   1.2631e-16   1.4496e-18   2.8392e-02   7.5231e-04
}\loadedtable

\begin{figure}
\begin{tikzpicture}
\begin{semilogyaxis}[width=\textwidth, height=0.4\textheight, xlabel={$n$}, ylabel={$\frac{\norm{\tilde{T}-X}_2}{\norm{\tilde{T}}_2}$}, legend pos=south west]

\addplot table [x=n,y=reldist] \loadedtable;
\addplot table [x=n,y=reldistwilk] \loadedtable;
\legend{Random polynomial, Wilkinson polynomial};
\end{semilogyaxis}
\end{tikzpicture}
\caption{Relative distance from structure} \label{tbl:reldist}
\end{figure}
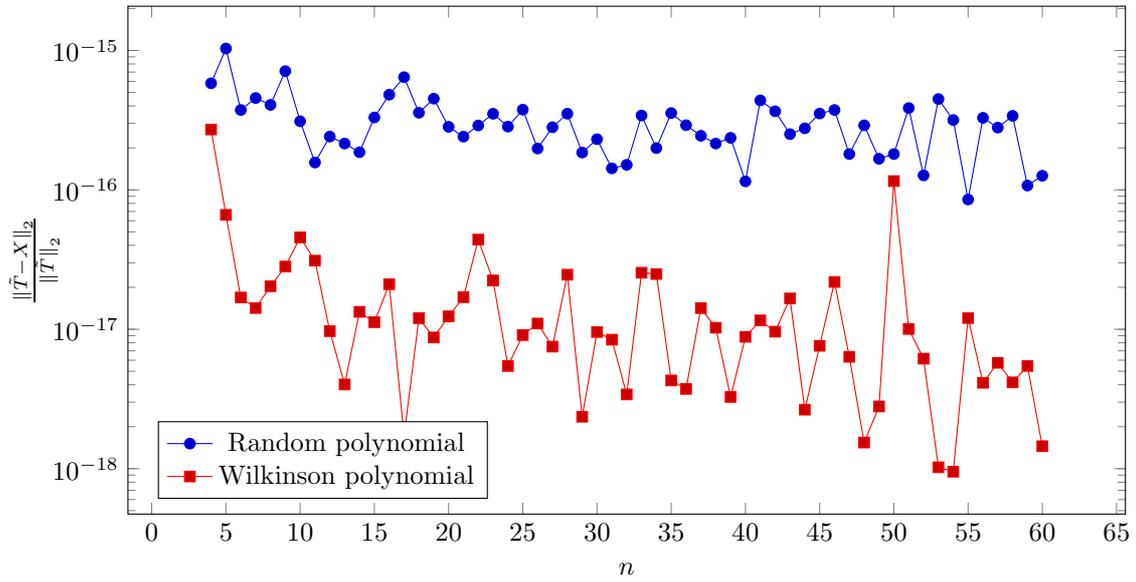

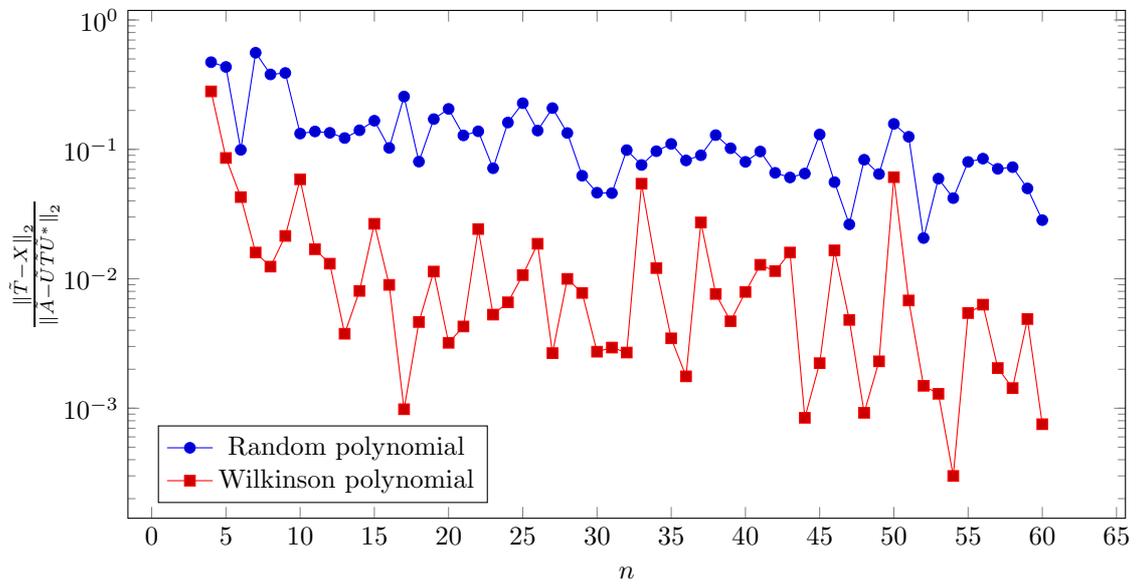
\begin{figure}
\begin{tikzpicture}
\begin{semilogyaxis}[width=\textwidth, height=0.4\textheight, xlabel={$n$}, ylabel={$\frac{\norm{\tilde{T}-X}_2}{\norm{\tilde{A}-\tilde{U}\tilde{T}\tilde{U}^*}_2}$}, legend pos=south west]

\addplot table [x=n,y=percent] \loadedtable;
\addplot table [x=n,y=percentwilk] \loadedtable;
\legend{Random polynomial, Wilkinson polynomial};
\end{semilogyaxis}
\end{tikzpicture}
\caption{Ratio of error due to structure loss} \label{tbl:percent}
\end{figure}

\section{Conclusions}
\label{sec:conc}
We have provided explicit conditions under which a matrix is unitary (resp.\ Hermitian) plus low rank,
and have given a construction for the closest unitary (resp.\ Hermitian) plus rank $k$ to a given matrix $A$,
in both the spectral and the Frobenius norm.

We have presented an algorithm based on the Lanczos iteration to construct explicitly, 
given a matrix $A \in \mathcal{H}_k$, (where $k$ is not known a priori),
a representation of the form $A=H+GB^*$, where $H$ is Hermitian and $GB^*$ is a rank-$k$ correction.
A variant for unitary-plus-low-rank matrices based on the Golub-Kahan bidiagonalization scheme has been presented as well. We tested these two algorithms on various examples coming from applications in numerical linear algebra, including a large-scale example.


\end{document}